\documentclass[12pt]{amsart}
\usepackage[dvips]{graphics}
\usepackage{amsfonts}
\usepackage{amssymb}
\usepackage{a4}

\DeclareMathAlphabet{\eusm}{U}{}{}{}  
\SetMathAlphabet\eusm{normal}{U}{eus}{m}{n}
\SetMathAlphabet\eusm{bold}{U}{eus}{b}{n}

\DeclareMathAlphabet{\eufrak}{U}{}{}{}  
\SetMathAlphabet\eufrak{normal}{U}{euf}{m}{n}
\SetMathAlphabet\eufrak{bold}{U}{euf}{b}{n}

\newtheorem{theorem}{Theorem}[section]
\newtheorem{proposition}[theorem]{Proposition}
\newtheorem{lemma}[theorem]{Lemma}
\newtheorem{corollary}[theorem]{Corollary}

\theoremstyle{definition}
\newtheorem{definition}[theorem]{Definition}
\newtheorem{example}[theorem]{Example}

\theoremstyle{remark}
\newtheorem{remark}[theorem]{Remark}

\numberwithin{equation}{section}

\newcommand{\1}{{\bf 1}}
\newcommand{\A}{\mathcal{A}}
\newcommand{\ad}{{\rm ad}}
\newcommand{\C}{\mathbb{C}}
\newcommand{\e}{\varepsilon}
\newcommand{\E}{\mathcal{E}}
\newcommand{\G}{\mathbb{G}}
\newcommand{\id}{{\rm id}}
\newcommand{\op}{{{\rm op}}}
\newcommand{\Q}{\mathcal{Q}}
\newcommand{\R}{\mathbb{R}}

\begin{document}
\title[Symmetries]{Symmetries of L\'evy processes\\ 
on compact quantum groups, \\
their Markov semigroups \\
and potential theory}

\author{Fabio Cipriani}
\address{FC: Dipartimento di Matematica, Politecnico di Milano, Piazza Leonardo da Vinci
32, 20133 Milan, Italy}
\email{fabio.cipriani@polimi.it}

\author{Uwe Franz}
\address{UF: D\'epartement de math\'ematiques de Besan\c{c}on,
Universit\'e de Franche-Comt\'e 16, route de Gray, 25 030
Besan\c{c}on cedex, France}
\email{uwe.franz@univ-fcomte.fr}
\urladdr{http://lmb.univ-fcomte.fr/uwe-franz}

\author{Anna Kula}
\address{AK: Instytut Matematyczny, Uniwersytet Wroc{\l}awski,  pl. Grunwaldzki
2/4, 50-384 Wroc³aw, Poland
\ and \ Instytut Matematyki, Jagiellonian University,  ul. {\L}ojasiewicza 6, 30
348 Krak{\'o}w, Poland
}
\email{anna.kula@math.uni.wroc.pl}

\begin{abstract}
Strongly continuous semigroups of unital completely positive maps (i.e.\ quantum
Markov semigroups or quantum dynamical semigroups) on compact quantum groups are
studied. We show that quantum Markov semigroups on the universal or reduced
C${}^*$-algebra of a compact quantum group which are translation invariant
(w.r.t.\ to the coproduct) are in one-to-one correspondence with L\'evy
processes on its $*$-Hopf algebra. We use the theory of L\'evy processes on
involutive bialgebras to characterize symmetry properties of the associated
quantum Markov semigroup. It turns out that the quantum Markov semigroup is
GNS-symmetric (resp.\ KMS-symmetric) if and only if the generating functional
of the L\'evy process is invariant under the antipode (resp.\ the unitary
antipode). Furthermore, we study L\'evy processes whose marginal states are
invariant under the adjoint action. In particular, we give a complete
description of generating functionals on the free orthogonal quantum group
$O_n^+$ that are invariant under the adjoint action. Finally, some aspects of
the potential theory are investigated. We describe how the Dirichlet form and a
derivation can be recovered from a quantum Markov semigroup and its L\'evy
process and we show how, under the assumption of GNS-symmetry and using the associated Sch\"urmann triple, this gives rise
to spectral triples. We discuss in details how the above results apply to
compact groups, group C$^*$-algebras of countable discrete groups, free orthogonal quantum
groups $O_n^+$ and the twisted $SU_q (2)$ quantum group.
\end{abstract}

\maketitle

\setcounter{tocdepth}{1}
\tableofcontents

\section{Introduction}
\label{sec-intro}

Quantum Markov semigroups, i.e. semigroups of contractive, completely positive maps,
are mathematical models for open quantum systems. In this paper we study a class of such
semigroups on the C${}^*$-algebras of compact quantum groups (also called
Woronowicz C${}^*$-algebras).

A compact quantum group $\mathbb{G}$ is a unital C${}^*$-algebra C$(\mathbb{G})$
equipped with additional
structure, that generalizes the C${}^*$-algebra of continuous functions on a
compact group (see Section \ref{sec-prelim} for the precise definition). In
particular, any commutative compact quantum group is isomorphic to the
C${}^*$-algebra of continuous functions on a compact group.
The quantum group structure thus plays two roles: on one hand, positive
functionals on C$(\mathbb{G})$ replace
the states of a classical Markov process and, on the other hand, the actions of C$(\mathbb{G})$ on
itself, allow us to formulate important symmetry properties for quantum Markov semigroups.

We show that quantum Markov semigroups on the (reduced or universal)
C${}^*$-algebra of a compact quantum group that are translation invariant
w.r.t. the coproduct are in one-to-one correspondence with L\'evy processes on
its $*$-Hopf algebra $\A$, see Theorems \ref{extension-to-reduced} and
\ref{prop-trans-inv}. This shows that the characterisation of L\'evy processes in topological groups as the Markov processes which are invariant under time and space translations extends to compact quantum groups. In particular, if the compact quantum group is commutative, the associated stochastic processes reduce to L\'evy processes with values in a compact group, i.e., stochastic processes with stationary and independent increments.

In general, a KMS-symmetry property of a quantum Markov semigroup on a
C$^*$-algebra with respect to given KMS state on it,
allows to study the semigroup on the scale of associated noncommutative
$L_p$-spaces. On compact quantum groups the natural
state to refer to is the unique translation invariant state: it is called the
Haar state because it reduces to the Haar measure of compact group when
C$(\mathbb{G})$ is commutative.
In Section \ref{sec_gns_kms} we show that the quantum Markov semigroup is
KMS-symmetric (with respect to the Haar state) if and only if the generating
functional of its associated L\'evy process is invariant under the unitary
antipode, and that the quantum Markov semigroup satisfies the stronger condition
of GNS-symmetry if and only if the generating functional of its associated
L\'evy process is invariant under the antipode.

In Section \ref{sec_schurmann} we characterize the Sch\"urmann triples of KMS-symmetric L\'evy processes.

In the classical literature on Brownian motion or L\'evy processes on (simple) Lie groups, the analysis of their invariance under the adjoint action of the group on itself has been particularly intense. We formulate
this invariance property for compact quantum groups and show that it imposes a
very strong restriction. In Section \ref{sec_adinv} we develop a method that allows to determine
ad-invariant generating functionals on compact quantum groups of Kac type, i.e. when the Haar state is a trace.
Using this method we find a complete classification of the ad-invariant L\'evy processes on the free orthogonal quantum groups in Section \ref{sec-On}.

In Section 7 we give a complete description of the Dirichlet form associated to
KMS symmetric L\'evy processes and, in the GNS symmetric case,
a characterization of the associated quadratic forms on the Hopf algebra $\A$, arising in this way,
in terms of their translation invariance.

In the framework of Alain Connes' noncommutative geometry \cite{connes94},
efforts have been directed towards the construction and investigation
of Dirac operators and spectral triples on the Woronowicz quantum groups $SU_q (2)$ and related homogeneous noncommutative
spaces (see for example \cite{chakraborty+pal03a}, \cite{chakraborty+pal03b},
\cite{connes04a}, \cite{connes04b}, \cite{dabrowski+landi+all05}).
The relevance of this point of view relies in the fact that a spectral
triple allows to construct topological invariants as cyclic cocycles in
cyclic cohomology and local  couplings with the K-theory of the C$^*$-algebra (Connes-Chern character).

In Section 8 we construct a Hilbert bimodule derivation, giving rise to differential calculus on the compact quantum group C$^*$-algebra C$(\mathbb{G})$,
which, in the GNS symmetric case, allows to represent the Dirichlet form as a generalized Dirichlet integral
\[
\E[a]=\frac{1}{2}\|da\|^2\, .
\]
Using the derivation, we then construct a Dirac operator $D$ whose spectrum is explicitly  determined by the spectrum of the Dirichlet form
on the GNS Hilbert space of the Haar state. Later we show that the Dirac operator $D$ is part of spectral triple with respect to which the elements of the Hopf algebra $\A$
are Lipschitz.

In the last three Sections 9, 10, 11, we discuss in detail examples of the above constructions on
compact Lie groups, group C$^*$-algebras of countable discrete groups, the free orthogonal
quantum groups $O_N^+$ and the Woronowicz quantum groups $SU_q (2)$.

\section{Preliminaries}
\label{sec-prelim}

 Sesquilinear forms will be linear in the right entry. For an algebra $\A$,
$\A'$ will denote the algebraic dual of $\A$, i.e. the space of all linear
functionals from $\A$ to $\C$. For a C${}^*$-algebra $\mathsf{A}$, by
$\mathsf{A}'$ we will mean the dual space of all linear continuous functionals
on $\mathsf{A}$. The symbol $\otimes$ will denote the spatial tensor
product of C${}^*$-algebras and $\odot$ the algebraic tensor product, see, e.g.,
\cite{pedersen79} for tensor products and other facts about C${}^*$-algebras.

\subsection{Compact Quantum Groups}

The notion of compact quantum groups has been introduced in
\cite{woronowicz87}. Here we adopt the definition from \cite{woronowicz98}
(Definition 1.1 of that paper).

\begin{definition}
A \emph{C${}^*$-bialgebra} (a compact quantum semigroup) is a pair $(\mathsf{A}, \Delta)$, where
$\mathsf{A}$ is a unital C${}^*$-algebra,
 $\Delta:\mathsf{A} \to \mathsf{A} \otimes \mathsf{A}$ is a unital,
 $*$-homomorphic map which is coassociative, i.e.\
\[ (\Delta \otimes \textup{id}_{\mathsf{A}})\circ \Delta = (\textup{id}_{\mathsf{A}} \otimes \Delta)
\circ\Delta.
\]
If the quantum cancellation properties
\[
\overline{{\rm Lin}}((1\otimes \mathsf{A})\Delta(\mathsf{A}) ) = \overline{{\rm Lin}}((\mathsf{A} \otimes 1)\Delta(\mathsf{A}) )
= \mathsf{A} \otimes \mathsf{A},
\]
are satisfied, then the pair $(\mathsf{A}, \Delta)$ is called a \emph{compact
  quantum group} (CQG).
\end{definition}

If the algebra $\mathsf{A}$ of a compact quantum group is commutative, then
$\mathsf{A}$ is isomorphic to the algebra $C(G)$ of continuous functions on a
compact group $G$. To emphasis that for an arbitrary (i.e.\ not necessarily
non-commutative) compact quantum group $(\mathsf{A},\Delta)$ the algebra
$\mathsf{A}$ replaces the algebra of continuous functions on an (abstract)
quantum analog of a group, the notation $\G=(\mathsf{A},\Delta)$ and
$\mathsf{A}=C(\G)$ is also frequently used.

The map $\Delta$ is called the \emph{coproduct} of $\mathsf{A}$ and it induces
the convolution product of functionals
\[
\lambda\star \mu:=(\lambda\otimes \mu)\circ\Delta, \;\;\;
\lambda,\mu\in\mathsf{A}'.
\]

The following fact is of fundamental importance, cf.\ \cite[Theorem
2.3]{woronowicz98}.

\begin{proposition} \label{prop-haar}
Let $\mathsf{A}$ be a compact quantum group. There exists a unique state $h \in \mathsf{A}'$ (called the \emph{Haar state} of
$\mathsf{A}$) such that for all $a \in \mathsf{A}$
\[
(h \otimes \textup{id}_{\mathsf{A}})\circ  \Delta (a)  = h(a) 1=
(\textup{id}_{\mathsf{A}} \otimes h)\circ  \Delta (a).
\]
\end{proposition}

The left (resp. right) part of the equation above are usually referred to as
left- (resp. right-) invariance property of the Haar state. In general, the
Haar state of a compact quantum group need not be faithful or tracial.

\subsection{Corepresentations}

An element $u =(u_{jk})_{1\le j,k\le n}\in M_n(\mathsf{A})$ is called an
{\emph{$n$-dimensional corepresentation of}} $\mathbb{G}=(\mathsf{A},\Delta)$ if
for all $j,k=1,\ldots,n$ we have $\Delta(u_{jk}) = \sum_{p=1}^n u_{jp} \otimes
u_{pk}$.  All corepresentations considered in this paper are supposed to be
finite-dimensional. A corepresentation $u$ is said to be \emph{non-degenerate}, if $u$ is
invertible, \emph{unitary}, if $u$ is unitary, and \emph{irreducible}, if the
only matrices $T\in M_n(\mathbb{C})$ with $Tu=uT$ are multiples of the identity
matrix. Two corepresentations $u,v\in M_n(\mathsf{A})$ are called
\emph{equivalent}, if there exists an invertible matrix $U\in M_n(\mathbb{C})$ such that $Uu=vU$.

An important feature of compact quantum groups is the existence of the dense
$*$-subalgebra $\mathcal{A}$ (the algebra of the \emph{polynomials} of $\mathsf{A}$), which is in fact a
Hopf $*$-algebra -- so for example $\Delta: \mathcal{A} \to \mathcal{A} \odot
\mathcal{A}$. With the notation $\G=(\mathsf{A},\Delta)$, the $*$-algebra
$\mathcal{A}$ is often denoted in the literature as ${\rm Pol}(\G)$.

Fix a complete family $(u^{(s)})_{s\in\mathcal{I}}$ of mutually
inequivalent irreducible unitary corepresentations of $\mathsf{A}$, then $\{u^{(s)}_{k\ell};
s\in\mathcal{I},1\le k,\ell\le n_s\}$ (where $n_s$ denotes the dimension of $u^{(s)}$) is
a linear basis of $\mathcal{A}$, cf.\ \cite[Proposition
5.1]{woronowicz98}. We shall reserve the index $s=0$ for the trivial
corepresentation $u^{(0)}=\mathbf{1}$.
The Hopf algebra structure on $\A$ is defined by
$$\e(u^{(s)}_{jk})=\delta_{jk}, \quad S(u^{(s)}_{jk})=(u^{(s)}_{kj})^*\quad \mbox{for} \;s\in\mathcal{I},\,j,k=1,\ldots,n_s,$$
where $\e:\A \to \C$ is the counit and $S:\A \to \A$ is the antipode. They satisfy
\begin{eqnarray}
  (\id \otimes \e) \circ \Delta  &=&\id= (\e \otimes \id) \circ \Delta, \label{counit_property}\\
m_\A \circ (\id \otimes S) \circ \Delta &=& \e(a)\1 = m_\A \circ (S \otimes \id) \circ \Delta,\label{antipode_property} \\
\big(S(a^*)^*\big) &=& a
\end{eqnarray}
for all $a\in \A$.
Let us also remind that the Haar state is always faithful on $\mathcal{A}$.

Set $V_s={\rm span}\,\{u^{(s)}_{jk};1\le j,k\le n_s\}$ for $s\in\mathcal{I}$. By
  \cite[Proposition 5.2]{woronowicz98}, there exists an irreducible unitary
  corepresentation $u^{(s^{\rm c})}$, called the \emph{contragredient}
  representation of $u^{(s)}$, such that $V_{s}^*=V_{s^{\rm c}}$. Clearly
  $(s^{\rm c})^{\rm c}=s$.

We shall frequently use \emph{Sweedler notation} for the coproduct of an element $a\in\mathcal{A}$, i.e.\ omit the summation and the index in the formula $\Delta(a)=\sum_{i} a_{(1),i}\otimes a_{(2),i}$ and write simply $\Delta(a)=a_{(1)}\otimes a_{(2)}$.

\subsection{The dual discrete quantum group}
\label{sec_dual}

To every compact quantum group $\mathbb{G}=(\mathsf{A},\Delta)$ there exists a
dual discrete quantum group $\hat{\mathbb{G}}$, cf.\ \cite{podles+woronowicz90}.
For our purposes it will be most convenient to introduce $\hat{\mathbb{G}}$
in the setting of Van Daele's algebraic quantum groups, cf.\
\cite{vandaele98,vandaele03}. However, the reader should be aware that we adopt
a slightly different convention for the Fourier transform.

\medskip
 A pair $(A, \Delta)$, consisting of a $*$-algebra $A$ (with or without
identity) and a coassociative comultiplication $\Delta:A \to M(A \odot A)$, is
called an \emph{algebraic quantum group} if the product is non-degenerate
(i.e. $ab=0$ for all $a$ implies $b=0$), if the two operators $T_1: A \odot A
\ni a\otimes b \mapsto \Delta(a)(b\otimes 1)\in A \odot A$ and $T_2: A \odot A
\ni a\otimes b \mapsto \Delta(a)(1\otimes b)\in A \odot A$ are well-defined
bijections and if there exists a nonzero left-invariant positive
functional on $A$. Here, $M(B)$ denotes the set of multipliers on $B$. We refer
the reader to \cite{vandaele98} for proofs and further details and will just recall a few facts here that we shall need later.

If $(\mathsf{A}, \Delta)$ is a compact quantum group then $(\A, \Delta|_\A)$ is
an algebraic quantum group (``of compact type'') and the Haar state is a
faithful left- and right-invariant functional.

For $a\in\mathcal{A}$ we can define $h_a\in \A'$ by the formula
$$ h_a(b) = h(ab) \quad \mbox{ for } b\in\A, $$
where $h$ is the Haar state, and we denote by $\hat{\A}$ the space of linear functionals on $\A$ of the form $h_a$ for $a\in \A$.

The set $\hat{\A}$ becomes an associative $*$-algebra with the convolution of
functionals as the multiplication: $ \lambda \star \mu = (\lambda \otimes \mu)
\circ \Delta$, and the involution $\lambda^*(x)=\overline{\lambda(S(x)^*)}$
($\lambda,\mu\in \hat{\A}$). Note that $\hat{\A}$ is closed under the
convolution by \cite[Proposition 4.2]{vandaele98}.
The Hopf structure is given as follows: the coproduct
$\hat{\Delta}$ is the dual of the product on $\A$, the antipode $\hat{S}$ is the
dual to $S$ and the counit $\hat{\e}$ is the evaluation in $\1$. In particular,
we have $\hat{S}(\lambda)(x) = \lambda(Sx)$ for $\lambda \in \hat{\A}$, $x\in
\A$ and
if $\hat{\Delta}(\lambda) \in \hat{\A} \odot \hat{\A}$ then
$$ \hat{\Delta}(\lambda) (x\otimes y) = \lambda_{(1)} (x) \otimes \lambda_{(2)}(y) = \lambda (xy), \quad x,y\in \A.$$
The pair $\hat{\mathbb{G}}=(\hat{\A},\hat{\Delta})$ is an algebraic quantum group, called the \emph{dual} of $\mathbb{G}$.

The linear map which associates to $a\in \A$ the functional $h_a\in \hat{\A}$
is called the \emph{Fourier transform}. Let us note that, due to the faithfulness of the Haar
state $h$, $\hat{\A}$ separates the points of $\A$.

\subsection{Woronowicz characters and modular automorphism group}
\label{woronowicz_characters}

A nice introduction to this part can be found in \cite{woronowicz87,woronowicz98},
\cite{klimyk+schmuedgen97} or \cite{timmermann08}.

For $a\in \mathsf{A}$, $\lambda\in \mathsf{A}'$ we define
\begin{eqnarray*}
\lambda \star a &=& ({\rm id}\otimes \lambda)\Delta(a), \\
a \star \lambda &=& (\lambda\otimes {\rm id})\Delta(a).
\end{eqnarray*}
If $a \in \A$ and $\lambda \in \A'$, then $\lambda \star a, a \star \lambda \in \A$.

For a compact quantum group $\mathsf{A}$ with dense $*$-Hopf algebra $\A$, there exists a unique family $(f_z)_{z\in \C}$ of linear multiplicative functionals on $\A$, called \emph{Woronowicz characters} (cf. \cite[Theorem 1.4]{woronowicz98}), such that
\begin{enumerate}
 \item $f_z(\1)=1$ for $z\in \C$,
 \item the mapping $\C\ni z \mapsto f_z(a) \in \C$ is an entire holomorphic function for all $a\in \A$,
 \item $f_0=\e$ and $f_{z_1}\star f_{z_2}= f_{z_1+z_2}$ for any $z_1, z_2 \in
\C$,
 \item $f_z(S(a)) = f_{-z}(a)$ and $f_{\bar{z}}(a^*) =
\overline{f_{-z}(a)}$ for any $z \in \C$, $a \in \A$,
 \item $S^2(a) = f_{-1}\star a \star f_1$ for $a\in \A$,
 \item the Haar state
$h$ satisfies:
$$ h(ab) = h(b(f_1 \star a \star f_1)), \quad a,b \in \A.$$
\end{enumerate}
The formulas
\begin{equation}\label{eq-mod-gr-Wor-ch}
\rho_{z,z'}(a) = f_{z} \star a \star f_{z'}, \quad \sigma_z=\rho_{iz,iz} \quad
\mbox{and} \quad \tau_z=\rho_{iz,-iz}
\end{equation}
define automorphisms of $\A$, in terms of which
$\sigma_t=\rho_{it,it}$ and $\tau_t=\rho_{it,-it}$, $t\in \R$, define one
parameter groups of automorphisms of $\A$. The former is known as \emph{modular
automorphism group}. Moreover, $h$ is the $(\sigma,-1)$-KMS state, which means
that it satisfies
\begin{equation} \label{KMS-state}
 h(ab) =h(b\sigma_{-i}(a)), \quad a,b\in \A,
\end{equation}
cf. \cite[Definition 5.3.1]{bratteli+robinson97} or \cite[Section
8.12]{pedersen79}. For $z,z'\in \C$,
$h(\rho_{z,z'}(a))=h(a)$, so
\begin{equation} \label{sigma-invariance}
h(\sigma_z(a))=h(\tau_z(a))=h(a), \quad a\in \A.
\end{equation}

The matrix elements of the irreducible unitary corepresentations satisfy the
famous generalized Peter-Weyl orthogonality relations
\begin{equation}\label{eq-pw}
h\left(\big(u^{(s)}_{ij}\big)^* u^{(t)}_{k\ell}\right) = \frac{\delta_{st}
\delta_{j\ell} f_{-1}\big (u^{(s)}_{ki}\big)}{D_s}, \qquad h\left(u^{(s)}_{ij} \big(u^{(t)}_{k\ell}\big)^*\right) = \frac{\delta_{st}
\delta_{ik} f_{1}\big (u^{(s)}_{\ell j}\big)}{D_s},
\end{equation}
where $f_1:\mathcal{A}\to\mathbb{C}$ is the Woronowicz character and
\[
D_s= \sum_{\ell=1}^{n_s} f_1\left(u_{\ell\ell}^{(s)}\right)
\]
is the quantum dimension of $u^{(s)}$, cf.\
\cite[Theorem 5.7.4]{woronowicz87}. Note that unitarity implies that the matrix
\[
\left(f_1\big((u^{(s)}_{jk})^*\big)\right)_{1 \le j,k\le n_s}\in M_{n_s}(\mathbb{C})
\]
is invertible, with inverse $\big(f_1(u^{(s)}_{jk})\big)_{jk}\in
M_{n_s}(\mathbb{C})$, cf.\ \cite[Equation (5.24)]{woronowicz87}.

\begin{remark} \label{rem_kac_type}
 The Haar state on a compact quantum group is a trace if and only if the
antipode is involutive, i.e.\ we have $S^2(a)=a$ for all $a\in\mathcal{A}$. In
this case we say that $(\mathsf{A},\Delta)$ is of Kac type. This is also
equivalent to the following equivalent conditions, cf.\ \cite[Theorem
1.5]{woronowicz98},
\begin{enumerate}
\item
$f_z=\varepsilon$ for all $z\in\mathbb{C}$,
\item
$\sigma_t={\rm id}$ for all $t\in\mathbb{R}$.
\end{enumerate}
\end{remark}

 The antipode $S$ and the automorphism $\tau_{\frac{i}{2}}$ are
closable operators on $\mathsf{A}$, and the closure $\overline{S}$ admits the
polar decomposition
\begin{equation} \label{decomposition_of_s}
 \overline{S} = R \circ T,
\end{equation}
where $T$ is the closure of $\tau_{\frac{i}{2}}$ and
$R:\mathsf{A}\to \mathsf{A}$ is a linear antimultiplicative norm preserving involution
 that commutes with hermitian conjugation and with
the semigroup $(\tau_t)_{t\in \R}$, i.e.\ $\tau_t \circ R = R \circ \tau_t$ for
all $t \in \R$, see \cite[Theorem 1.6]{woronowicz98}. The operator $R$ is called
the \emph{unitary antipode} and is related to Woronowicz characters through the
formula
\begin{equation}\label{eq-R}
R(a) = S(f_{\frac{1}{2}}\star a\star f_{-\frac{1}{2}}) \quad
\mbox{for} \quad a\in\mathcal{A}.
\end{equation}

\subsection{L\'evy processes on involutive bialgebras}
\label{processes}

 We recall the definition of L\'evy processes on $*$-bialgebras, cf.\
\cite{schuermann93}. An introduction to this topic can also be found in
\cite{franz06}. Lindsay and Skalski have developped an analytic theory of L\'evy
processes on C${}^*$-bialgebras, see \cite{lindsay+skalski12} and the references
therein.

\begin{definition}\label{def levy}
A family of unital $*$-homomorphisms $(j_{st})_{0\le s\le t}$ defined on a
$*$-bialgebra $\A$ with values in a unital $*$-algebra ${\mathcal B}$
with some fixed state $\Phi:{\mathcal B}\to \C$ is called a \emph{L\'evy process on $\A$}
(w.r.t.\ $\Phi$), if the following conditions are satisfied:
\begin{itemize}
\item[$(i)$]
the images corresponding to disjoint time intervals commute, i.e. \linebreak
$[j_{st}(\A),j_{s't'}(\A)]=\{0\}$ for $0\le s\le t\le s'\le
t'$, and expectations corresponding to disjoint time intervals
factorize, i.e.\
\[
\Phi(j_{s_1t_1}(a_1) \cdots j_{s_nt_n}(a_n)) =
\Phi(j_{s_1t_1}(a_1))\cdots \Phi(j_{s_nt_n}(a_n)),
\]
for all $n\in\mathbb{N}$, $a_1,\ldots,a_n\in\A$ and $0\le s_1\le
t_1\le\ldots\le t_n$;
\item[$(ii)$]
$m_{\mathcal B}\circ(j_{st}\otimes j_{tu})\circ \Delta = j_{su}$ for all $0\le s\le t\le
u$, where $m_{\mathcal B}$ denotes the multiplication of ${\mathcal B}$ and $\Delta$ is the comultiplication on $\A$;
\item[$(iii)$]
the functionals $\varphi_{st}=\Phi\circ j_{st}:\A\to \C$ depend
only on $t-s$;
\item[$(iv)$] $\displaystyle\lim_{t\searrow s} j_{st}(a) = j_{ss}(a) = \varepsilon(a)
1_{\mathcal B}$ for all $a\in {\A}$, where $1_{\mathcal B}$ denotes the unit of ${\mathcal B}$.
\end{itemize}
\end{definition}

We do not distinguish two L\'evy processes on the same $*$-bialgebra $\A$ which
are \emph{equivalent}. By this we mean that two processes $(j_{st})_{0\le s\le
t}$ and $(k_{st})_{0\le s\le t}$ with values in unital $*$-algebras $({\mathcal
B}, \Phi)$ and $({\mathcal B}', \Phi')$, respectively,
agree on all their finite joint moments, i.e.
$$ \Phi \big(j_{s_1t_1} (b_1 ) \cdots j_{s_nt_n} (b_n )\big) = \Phi'
\big(k_{s_1t_1} (b_1 ) \cdots k_{s_nt_n} (b_n )\big),$$
for all $n \in \mathbb{N}$, $s_1\le t_1 ,\ldots , s_n\le t_n$ and $b_1,
\ldots, b_n \in \mathcal{A}$.

If $(j_{st})_{0\le s\le t}$ is a L\'evy process, then the functionals
$\varphi_{t} :=\varphi_{0,t} = \varphi_{s,t+s}$ ($t\geq 0$) form a convolution
semigroup of states, i.e.
\begin{itemize}
\item $\varphi_0=\e$, $\varphi_s \star \varphi_t = \varphi_{s+t}$, $\lim_{t\to 0} \varphi_t(a) = \e(a)$ for all $a\in \A$,
\item $\varphi_t(\1)=1$, $\varphi_t(a^*a)\geq 0$ for all $a\in \A$ and $t\geq 0$.
\end{itemize}
 For such a semigroup there exists a linear functional $\phi$ which is hermitian
(i.e. $\phi (a^*)=\overline{\phi (a)}$ for $a\in \A$), conditionally positive
($\phi(a^*a) \geq 0$ when $a\in \ker \e$), vanishes on $\1$, and is such that
\begin{equation} \label{semigroup}
\varphi_t=\exp_\star t\phi = \varepsilon +t\phi + \frac{t^2}{2} \phi\star\phi
+ \cdots + \frac{t^n}{n!} \phi^{\star n} + \cdots.
\end{equation}
Conversely, by the Schoenberg correspondence (cf. \cite{franz06}), for every hermitian conditionally
positive linear functional $\phi:\A\to \C$ with $\phi(\1)=0$ there
exists a unique convolution semigroup of states $(\varphi_t)_{t \geq 0}$ which
satisfies \eqref{semigroup} and a unique (up to
equivalence) L\'evy process $(j_{st})_{0\le s\le t}$. The functional $\phi$
will be called the \emph{generating functional} of the L\'evy process
$(j_{st})_{0\le s\le t}$.

Given the convolution semigroup of states $(\varphi_t)_{t\geq 0}$,
we can also define the semigroup of operators on $\A$ (called the
\emph{Markov semigroup on $\A$} associated to $(j_{st})_{0\le s\le t}$)
$$T_t = (\id \otimes \varphi_t)\circ \Delta, \quad t\geq 0.$$

The infinitesimal generator of this semigroup is an operator $L:\A\to \A$, which
is related to $\phi$ by the relations
$$ L(a)=(\id \otimes \phi)\circ \Delta (a)=\phi \star a \quad \mbox{and} \quad \phi(a)=\e\circ L(a).$$
In this case we write $L=L_\phi$. As usual the formula to recover the semigroup from the generator is $T_t = \exp(tL)$ for $t\geq 0$. The fundamental theorem
of coalgebra ensures that all this makes sense in the bialgebra $\mathcal{A}$.

\medskip
Let $\mathcal{L}(\mathcal{A})$ denotes the algebra of linear operators from $\A$
to $\A$. Operators $L\in \mathcal{L}(\mathcal{A})$ of
the form $L=L_\phi=({\rm id}\otimes \phi)\circ \Delta$ for some linear
functional $\phi\in\mathcal{A}'$ will play an important role in the paper and we will refer to them as \emph{convolution operators}.

An operator $L\in \mathcal{L}(\mathcal{A})$ is a convolution operator
if and only if it is \emph{translation invariant} on $\A$, i.e.
\[
\Delta \circ L = ({\rm id}\otimes L)\circ\Delta,
\]
and, if this is the case, the linear functional $\phi$ can be recovered from $L$ using the formula
\[
\phi=\varepsilon\circ L.
\]

The map $\mathcal{A}'\ni \phi \to L_\phi \in \mathcal{L}(\mathcal{A})$ is
also
called \emph{the dual
  right representation}. It is a unital algebra homomorphism for the
convolution product, i.e.\ we have
\begin{eqnarray*}
L_\varepsilon &=& {\rm id}, \\
L_\phi\circ L_\psi &=& L_{\phi\star\psi},
\end{eqnarray*} for $\phi,\psi\in\mathcal{A}'$. Moreover, $L_\phi$ is hermitian, i.e.
\[
L_\phi (a^*)=(L_\phi a)^* \quad \mbox{for} \quad a\in \A
\]
iff $\phi$ is hermitian, i.e. $\phi(a^*)=\overline{\phi(a)}$, $a\in \A$.

\section{Translation invariant Markov semigroups}

Our goal is to construct Markov semigroups on compact quantum groups that
reflect the structure of the quantum group. In this section we show that it is
exactly the \emph{translation invariant} Markovian semigroups that can be
obtained from L\'evy processes on the algebra of smooth functions $\A={\rm
  Pol}(\mathbb{G})$ of the quantum group $\mathbb{G}=(\mathsf{A},\Delta)$.

For this purpose we first prove that the Markov  semigroup $(T_t)_{t\ge 0}$
of a L\'evy process on $\A$ has a unique extension to a strongly continuous
Markov semigroup on both its reduced and its universal C${}^*$-algebra. We then
show that the characterisation of L\'evy processes in topological groups as the
Markov processes which are invariant under time and space translations extends
to compact quantum groups.

\begin{definition} \label{def_markov}
 A strongly continuous semigroup of operators $(T_t)_{t\ge 0}$ on
a C$^*$-algebra $\mathsf{A}$ is called a \emph{quantum Markov semigroup on
$\mathsf{A}$} if every $T_t$ is a unital, completely positive contraction.
\end{definition}

If $(j_{st})_{0\le s \le t}$ is a L\'evy process on a $*$-bialgebra $\A$ with
the convolution semigroup of states $(\varphi_t)_{t\ge 0}$ on $\A$ and the
Markov semigroup $(T_t)_{t\ge 0}$ on $\A$, then, by a result of B\'edos, Murphy
and Tuset \cite[Theorem 3.3]{bedos+murphy+tuset01}, each $\varphi_t$
extends to a continuous functional on $\mathsf{A}_u$, the universal
C${}^*$-algebra generated by $\A$. Then the formula $T_t = (\id \otimes
\varphi_t) \circ \Delta$ makes sense on $\mathsf{A}_u$ (where
$\Delta:\mathsf{A}_u\to\mathsf{A}_u\otimes\mathsf{A}_u$ denotes the unique
unital $*$-homomorphism that extends
$\Delta:\mathcal{A}\to\mathcal{A}\otimes\mathcal{A}$) and one easily shows (in
the same way as in Proposition below) that $(T_t)_t$ becomes a Markov
semigroup
on $\mathsf{A}_u$ (in the sense of Definition \ref{def_markov}).

For us, however, it will be more natural to consider the reduced C${}^*$-algebra
generated by $\A$. This is the C${}^*$-algebra $\mathsf{A}_r$ obtained by
taking the norm closure of the GNS representation of $\mathcal{A}$ with
respect to the Haar state $h$.
The Haar state $h$ is by construction faithful on
$\mathsf{A}_r$. The coproduct on $\mathcal{A}$ extends to a unique unital
$*$-homomorphism $\Delta:\mathsf{A}_r\to\mathsf{A}_r\otimes\mathsf{A}_r$ which
makes the pair $(\mathsf{A}_r,\Delta)$ a compact quantum group. The following
result shows that, even though $\varphi_t:\mathcal{A}\to\mathbb{C}$ can be
unbounded with respect to the reduced C${}^*$-norm and therefore may not extend to
$\mathsf{A}_r$,  $(T_t)_{t\ge 0}$ always extends to a Markov
semigroup on $\mathsf{A}_r$.

Michael Brannan showed that states on any
$C^*$-algebraic version $C(\mathbb{G})$ of $\mathbb{G}$ define a continuous
convolution operator on the reduced version $C_r(\mathbb{G})$, cf.\
\cite[Lemma 3.4]{brannan2011}. We will need a similar result for convolution
semigroups of states on ${\rm Pol}(\mathbb{G})$.

\begin{theorem}
\label{extension-to-reduced}
Each L\'evy process $(j_{st})_{0\le s\le t}$ on the Hopf $*$-algebra $\A$ gives
rise to a unique strongly continuous Markov semigroup $(T_t)_{t\ge 0}$ on
$\mathsf{A}_r$, the reduced C${}^*$-algebra generated by $\A$.
\end{theorem}

\begin{proof}
 Let $(\lambda, \mathcal{H}, \xi)$ be the GNS representation of $\A$ for the
Haar state $h$, thus $h(a)=\langle \xi, \lambda(a) \xi\rangle $ for $a\in
\A$. 
We denote by $\|.\|_r$ the norm in $\mathsf{A}_r$, that is $\|a\|_r = \|\lambda(a)\|$, where $\|.\|$ denotes the operator norm.

Similarly, let $(\rho_t, \mathcal{H}_t, \xi_t)$ be the GNS representation of
$\A$ for the state $\varphi_t = \Phi \circ j_{0t}$, so that
$\varphi_t (a)= \langle \xi_t, \rho_t(a) \xi_t \rangle $ for $a\in \A$.

We define the operators
\begin{eqnarray*}
 i_t &: & \mathcal{H} \ni v \to v \otimes \xi_t \in  \mathcal{H}\otimes \mathcal{H}_t\\
 \pi_t &: & \mathcal{H}\otimes \mathcal{H}_t \ni v\otimes w \to \langle \xi_t,w \rangle_{\mathcal{H}_t} \, v \in  \mathcal{H}\\
 E_t &: &  B(\mathcal{H}\otimes \mathcal{H}_t) \ni X \to \pi_{t} \circ X \circ i_{t} \in B(\mathcal{H}).
\end{eqnarray*}
Since for each $t$, $i_t$ is an isometry
and $\pi_t$ is contractive,
$E_t$ is contractive too: $\|E_t (X)\| = \|\pi_{t} \circ X \circ i_{t}\| \leq \|X\|$.

Next we define
 $$ U: \lambda(\mathcal{A})\xi\otimes\rho_t(\mathcal{A})\xi_t \ni \lambda(a) \xi
\otimes \rho_t (b) \xi_t \mapsto \lambda (a_{(1)})\xi \otimes \rho_t (a_{(2)}b)
\xi_t \in \mathcal{H}\otimes \mathcal{H}_t $$
and we check that it is an isometry with adjoint given by
 $$U^*(\lambda(a) \xi \otimes \rho_t (b) \xi_t) = \lambda(a_{(1)}) \xi \otimes
\rho_t (S(a_{(2)})b) \xi_t.$$
Indeed, using the invariance of the Haar measure, we show that $U$ is isometric
\begin{eqnarray*}
 \lefteqn{\big\langle U(\lambda(a) \xi \otimes \rho_t (b) \xi_t) , U(\lambda(c)
\xi \otimes \rho_t (d) \xi_t) \big\rangle }\\
&=& \big\langle \lambda (a_{(1)})\xi \otimes \rho_t (a_{(2)}b) \xi_t,
\lambda (c_{(1)})\xi \otimes \rho_t (c_{(2)}d) \xi_t \big\rangle
\\ &=& h( a_{(1)}^* c_{(1)}) \varphi_t (b^*a_{(2)}^*c_{(2)}d)
= (h\otimes \varphi_t^{b,d}) ( a_{(1)}^* c_{(1)} \otimes a_{(2)}^*c_{(2)})
= (h\star \varphi_t^{b,d}) (a^*c)
\\ &=& h(a^*c) \varphi_t^{b,d}(\1) = h(a^*c) \varphi_t(b^*d)
= \big\langle \lambda(a) \xi \otimes \rho_t (b) \xi_t , \lambda(c) \xi \otimes
\rho_t (d) \xi_t \big\rangle,
\end{eqnarray*}
where $\varphi_t^{b,d} (x) := \varphi_t (b^*xd)$. Moreover, by the antipode
property \eqref{antipode_property} we have
\begin{eqnarray*}
 \lefteqn{UU^*(\lambda(a) \xi \otimes \rho_t (b) \xi_t) = U(\lambda(a_{(1)}) \xi \otimes \rho_t (S(a_{(2)})b) \xi_t) }\\
 &=& \lambda(a_{(1)}) \xi \otimes \rho_t (a_{(2)}S(a_{(3)})b) \xi_t
 = \lambda(a_{(1)}\e(a_{(2)})) \xi \otimes \rho_t (b) \xi_t =\lambda(a) \xi \otimes \rho_t (b) \xi_t,
\end{eqnarray*}
which implies that $U$ is an isometry with dense image and therefore extends to a unique unitary operator denoted again by $U$.

Now the fact that the Markov semigroup $(T_t)_t$ is bounded on $\mathsf{A}_r$, i.e.
$$\|T_t(a)\|_r = \|\lambda(T_t(a))\|_{B(\mathcal{H})} \leq \|\lambda(a)\|_{B(\mathcal{H})} =\|a\|_r,$$
follows immediately from the relation
\begin{equation} \label{markov_bounded}
 \lambda(T_t(a)) = E_t \big( U (\lambda(a) \otimes \id_{\mathcal{H}_t}) U^*
\big),
\end{equation}
since
\begin{eqnarray*}
\|\lambda(T_t(a))\|&=& \| E_t (U (\lambda(a) \otimes \id_{\mathcal{H}_t}) U^*)\| \leq
\|U (\lambda(a) \otimes \id_{\mathcal{H}_t}) U^*\| \\
&=& \|\lambda(a) \otimes \id_{\mathcal{H}_t}\|=\|\lambda(a)\|.
\end{eqnarray*}
To see that \eqref{markov_bounded} holds, let us fix $v\in \mathcal{H}$ and $b\in \A$ such that $v=\lambda(b)\xi$. Then
\begin{eqnarray*}
E_t (U (\lambda(a) \otimes \id_{\mathcal{H}_t}) U^*) v
&=& (\pi_{t} \circ U \circ (\lambda(a) \otimes \id_{\mathcal{H}_t})\circ  U^* \circ i_{t})(\lambda(b)\xi) \\
&=& (\pi_{t} \circ U \circ (\lambda(a) \otimes \id_{\mathcal{H}_t}) \circ  U^*)\, (\lambda(b)\xi \otimes \xi_t)  \\
&=& \pi_{t} \circ U \circ (\lambda(a) \otimes \id_{\mathcal{H}_t})\, \left(\lambda(b_{(1)})\xi \otimes \rho_t(S(b_{(2)}))\xi_t \right)\\
&=& \pi_{t} \circ U \, \left( \lambda(ab_{(1)})\xi \otimes \rho_t(S(b_{(2)}))\xi_t\right)  \\
&=& \pi_{t} \, \left( \lambda(a_{(1)}b_{(1)})\xi \otimes \rho_t(a_{(2)}b_{(2)}S(b_{(3)}))\xi_t \right) \\
&=& \pi_{t} \, \left( \lambda(a_{(1)}b)\xi \otimes \rho_t(a_{(2)})\xi_t \right) \\
&=& \langle \xi_t, \rho_t(a_{(2)})\xi_t \rangle \; \lambda(a_{(1)}b)\xi  \\
&=& \lambda(a_{(1)}\varphi_t(a_{(2)})) \lambda( b)\xi = \lambda(T_t(a))v.
\end{eqnarray*}

 This way we showed that each $T_t$ extends to a contraction on $\mathsf{A}_r$.
The extensions again form a semigroup and since both $\Delta$ and $\varphi_t$
are completely positive and unital, $T_t$ is too. Let us now check that
$(T_t)_t$ forms a strongly continuous semigroup on $\mathsf{A}_r$.

For a given $a\in \mathsf{A}_r$ we choose, by density, an element $b \in \A$ such
$\|a-b\|_r<\epsilon$. By definition for $b\in \A$, $T_t(b)= \varphi_t \star b =
(\id \otimes \varphi_t) \circ \Delta (b)$, where $(\varphi_t)_t$ is the
convolution semigroup of states on $\A$ (cf. Section \ref{processes}). Thus
\begin{eqnarray*}
 \|T_t(a)-a\|_r &\leq& \|T_t(a)-T_t(b)\|_r+\|T_t(b)-b\|_r+\|b-a\|_r\\
&\leq& 2\|a-b\|_r+\|(\varphi_t \star b)-b\|_r \leq 2\epsilon+\sum \|b_{(1)}\varphi_t(b_{(2)})-b_{(1)} \varepsilon(b_{(2)})\|_r \\
&=& 2\epsilon+ \sum |\varphi_t(b_{(2)})- \varepsilon(b_{(2)})| \|b_{(1)}\|_r .
\end{eqnarray*}
Since $\lim_{t\to 0+} \varphi_t(b)=\e(b)$ for any $b\in \A$ and the sum is finite, we conclude that
$$ \lim_{t\to 0+} \|T_t(a)-a\|_r=0 \quad \mbox{for each} \; a\in \mathsf{A}_r.$$
\end{proof}

The next results give a characterisation of Markov semigroups which are related
to L\'evy processes on compact quantum groups.
\begin{lemma} \label{lem_trans_invariant}
Let $(\mathsf{A},\Delta)$ be a compact quantum group and let
$T:\mathsf{A}\to\mathsf{A}$ be a completely bounded linear map.

If $T$ is translation invariant, i.e.\ satisfies
\[
\Delta\circ T= ({\rm id}\otimes T)\circ\Delta
\]
then $T(V_s)\subseteq V_s$ for all $s\in\mathcal{I}$ and therefore $T$ also
leaves the $*$-Hopf algebra $\mathcal{A}$ invariant.
\end{lemma}

\begin{proof}
Let $s,s'\in \mathcal{I}$, $s\not=s'$, and $1\le j,k\le n_s$, $1\le p,q\le
n_{s'}$. Since the Haar state is idempotent, we have
\begin{eqnarray*}
h\left(\left(u^{(s')}_{pq}\right)^*T\left(u^{(s)}_{jk}\right)\right) &=&
(h\star h)\left(\left(u^{(s')}_{pq}\right)^*T\left(u^{(s)}_{jk}\right)\right) \\
&=& \sum_{r=1}^{n_{s'}} (h\otimes h) \left( \left(\left(u^{(s')}_{pr}\right)^*\otimes
  \left(u^{(s')}_{rq}\right)^*\right)\Delta\left(T\left(u^{(s)}_{jk}\right)\right)\right)
\\
&=&
\sum_{r=1}^{n_{s'}} \sum_{\ell=1}^{n_{s}} (h\otimes h) \left( \left(u^{(s')}_{pr}\right)^*\otimes
  \left(u^{(s')}_{rq}\right)^*\left(u^{(s)}_{j\ell}\otimes
    T\left(u^{(s)}_{\ell k}\right)\right)\right)
\\
&=& \sum_{\ell=1}^{n_s} \delta_{ss'} \frac{\overline{f_1((u^{(s)}_{jp})^*)}}{D_s}
  h\left((u^{(s')}_{\ell q})^*T(u^{(s)}_{\ell k})\right),
\end{eqnarray*}
i.e.\ $h\left(\big(u^{(s')}_{pq}\big)^*T\big(u^{(s)}_{jk}\big)\right)=0$
for all  $s,s'\in \mathcal{I}$, with $s\not=s'$, and all $1\le j,k\le n_s$,
$1\le p,q\le n_{s'}$. Therefore $T\big( u^{(s)}_{jk}\big)\in V_s$.
\end{proof}

\begin{theorem}\label{prop-trans-inv}
Let $(\mathsf{A},\Delta)$ be a compact quantum group and
$(T_t)_{t\ge 0}$ a quantum Markov semigroup on $\mathsf{A}$.

Then $(T_t)_{t\ge 0}$ is the quantum Markov semigroup of a (uniquely
determined) L\'evy process on  $\A$ if and only if $T_t$ is
translation invariant for all $t\ge 0$.
\end{theorem}

\begin{proof}
If $(T_t)_{t}$ comes from a L\'evy process on $\A$, then, on
$\mathcal{A}$,  $T_t = (\id \otimes \varphi_t) \circ \Delta$ and so
$$ \Delta \circ T_t = (\id \otimes \id \otimes \varphi_t) \circ (\Delta \otimes
\id) \circ \Delta = (\id \otimes \big((\id \otimes \varphi_t) \circ \Delta\big)
) \circ \Delta = (\id \otimes T_t) \circ \Delta.
$$
Hence $T_t$ is translation invariant on $\A$, and therefore also on $\mathsf{A}$
by continuity.

Conversely, if every $T_t$ is translation invariant, then Lemma
\ref{lem_trans_invariant} implies that, for all $a\in V_s$, $T_ta\in V_s$ and
so, since $V_s$ is finite dimensional, $\e(T_ta) \to \e(a)$ as $t\to 0$. It now
follows easily that $\varphi_t:=\e\circ T_t|_{\A}$ defines a convolution
semigroups of states whose generating functional defines a L\'evy process whose
Markov semigroup is $(T_t)_t$.

\end{proof}

The corresponding result, for counital multiplier C$^*$-bialgebras satisfying a
residual vanishing at infinity condition, was proved by Lindsay and Skalski
\cite[Proposition 3.2]{lindsay+skalski11}). Their result covers coamenable
compact quantum groups (where the counit extends continuously to the
C$^*$-algebra). The above proof, for all compact quantum groups, is simpler.

\section{GNS-Symmetry and KMS-symmetry of convolution operators}
\label{sec_gns_kms}

In this section we study symmetry properties of convolution operators
$L_\phi(a)=\phi \star a$ on $\mathcal{A}$ and we show that they can be
translated into invariance properties of the corresponding generating functional
$\phi$.

We will use two antilinear involutions $\#$ and $\star$ on $\mathcal{A}'$, defined by
\begin{eqnarray*}
\phi^\#(a) &=& \overline{\phi(a^*)}, \\
\phi^\star(a) &=& \phi^\#\big(S(a)\big),
\end{eqnarray*}
for $a\in\mathcal{A}$. A functional $\phi\in\mathcal{A}'$ is hermitian if and
only if $\phi^\#=\phi$. Furthermore, we have
$\varepsilon^\#=\varepsilon^\star=\varepsilon$ and $h^\#=h^\star=h$.  Note that
$\#$ is multiplicative whereas $\star$ is anti-multiplicative with respect to
the convolution of functionals:
$$(\phi\star\psi)^\#=\phi^\#\star\psi^\#, \quad
(\phi\star\psi)^\star=\psi^\star\star\phi^\star$$
 for $\phi,\psi\in\mathcal{A}'$.

\medskip
Let us denote by $L^2(\mathsf{A},h)$ the GNS Hilbert space of $(\mathsf{A},h)$, by
$\xi_h =1_\mathsf{A}\in L^2 (\mathsf{A},h)$ the cyclic vector representing the Haar state:
$h(a)= \langle \xi_h , a\xi_h \rangle$ and let us assume that we are given an
embedding, i.e.\
an injective linear map $i:\A \to L^2(\mathsf{A},h)$ with a dense range. We say
that a linear operator $L:\A \to \A$ \emph{admits an $i$-adjoint} if there
exists $L^\dag: \A \to \A$ such that
\begin{equation*}
 \big\langle i(a), i(Lb) \big\rangle = \big\langle i(L^\dag a), i(b) \big\rangle
\end{equation*}
 for any $a,b\in \A$. Since $h$ is faithful on $\A$ and since $i$ has a dense
range, the adjoint is unique if it exists. Then, an operator $L\in \mathcal{L}
(\mathcal{A})$ is called \emph{$i$-symmetric} if $L$ equals to its $i$-adjoint.

 In this paper we shall consider two embeddings. The first one is the natural
inclusion coming from the GNS construction
$$i_h: \A \ni a \to a\xi_h \in L^2(\mathsf{A},h).$$
\label{gns_embedding}

\begin{definition}
 A map $L^\star\in\mathcal{L}(\mathcal{A})$ such that
\begin{equation} \label{gns_symm_condition}
h \big (a^* L(b) \big) = h \big( L^\star(a)^* b \big)
\end{equation}
for all $a,b,\in\mathcal{A}$ will be called a \emph{GNS-adjoint}, or simply \emph{adjoint}
of $L$ w.r.t.\ $h$. A map $L$ will be called \emph{GNS-symmetric} if $L=L^\star$.
\end{definition}

Let us observe that a convolution operator always admits a GNS-adjoint.
\begin{proposition}\label{prop-self-adj}
Let $\phi\in\mathcal{A}'$. Then there exists a unique convolution operator $L_\phi^\star$ that
is adjoint to $L_\phi$ w.r.t.\ the Haar state, i.e.\ that satisfies
\[
h\big(a^* L_\phi(b)\big) = h \big(L_\phi^\star(a)^*b\big)
\]
for all $a,b\in\mathcal{A}$. The adjoint of $L_\phi$ is given by
$$L^\star_\phi=L_{\phi^\star}.$$
Therefore $L_\phi$ is GNS-symmetric if and only if $\phi^\star=\phi$.
\end{proposition}
\begin{proof}
This is simply the fact that the dual right representation is a
$*$-repre\-sen\-ta\-tion w.r.t.\ to the involution $\star$ and the inner product
defined by the Haar state as $\mathcal{A}\times\mathcal{A}\ni(a,b)\mapsto\langle
a,b\rangle=h(a^*b)\in\mathbb{C}$. The proof is the same as in the finite-dimensional case,
see \cite[Proposition 2.3]{vandaele97}. See also \cite[Proposition 3.4]{franz+skalski-erg08}.
\end{proof}

\label{symmetric_embedding}
The second embedding we can consider is the \emph{symmetric embedding}
$$i_s:\A \ni a \mapsto i_s(a) = \sigma_{-\frac{i}{4}}(a)\xi_h \in L^2 (A,h)$$
and the related notion of symmetry is the following.
\begin{definition} \label{kms_adjoint_alg}
 We shall call a map $L^\flat\in\mathcal{L}(\mathcal{A})$ the \emph{KMS-adjoint}
of
  $L\in\mathcal{L}(\mathcal{A})$, if  we have
\begin{equation} \label{kms_symm_condition_alg}
h\big(\sigma_{-\frac{i}{2}}(a)^*L(b)\big) = h\big(
L^\flat(a)^*\sigma_{-\frac{i}{2}}(b)\big)
\end{equation}
for all $a,b,\in\mathcal{A}$. An operator $L\in\mathcal{L}(\mathcal{A})$ is
called
  \emph{KMS-symmetric} if $L^\flat=L$.
\end{definition}

 Let us note here that a definition of KMS-symmetric operator on a von Neumann
algebra was introduced by Goldstein and Lindsay (\cite{goldstein+lindsay95}) in
the framework of (noncommutative) Haagerup $L^p$-spaces and by Cipriani in his
PhD thesis (cf. \cite{cipriani97}) in the context of the standard form of von
Neumann algebras. Later, in \cite[Definition 2.31]{cipriani08}, a definition of
a KMS-symmetric operator on a
C${}^*$-algebra was provided.

 In the sequel, we shall also need the definition of KMS-symmetric operators on
the whole C${}^*$-algebra. It is stated as follows.
\begin{definition} \label{kms_symm_cstar}
A linear map $L:\mathsf{A}\to\mathsf{A}$ is called \emph{KMS-symmetric} w.r.t.\
$h$ with modular automorphism group $(\sigma_t)_{t\in \R}$, if
\begin{equation} \label{kms_symm_condition}
h \big(a L(b)\big)= h\big(\sigma_{\frac{i}{2}}(b)L(\sigma_{-\frac{i}{2}}(a))\big)
\end{equation}
for all $a,b$ in a dense $\sigma$-invariant $*$-subalgebra $B$ of the C${}^*$-algebra $\mathsf{A}$.
\end{definition}
 Note that a continuous map $L:\mathsf{A} \to \mathsf{A}$ is KMS-symmetric in
the sense of Definition \ref{kms_symm_cstar} if it is $\A$-invariant (i.e. $L
(\A) \subset \A$), hermitian and $(\sigma,-1)$-KMS-symmetric in the sense of
\cite[Definition 2.31]{cipriani08}. The temperature $\beta=-1$ is chosen
according to the KMS-property of the Haar state
$h(ab)=h(b\sigma_{-i}(a))$, see Equation \eqref{KMS-state}.

The analogue of Proposition \ref{prop-self-adj} for KMS-symmetric operators on $\A$ is now the following.
\begin{theorem}\label{thm-kms-adj}
Let $\phi\in\mathcal{A}'$. Then there exists a unique convolution operator $L^\flat_\phi$ that
is the KMS-adjoint to $L_\phi$ w.r.t.\ the Haar state, i.e.\ that satisfies
\[
h\big(\sigma_{-\frac{i}{2}}(a)^*L_\phi(b)\big) = h\big(
L_\phi^\flat(a)^*\sigma_{-\frac{i}{2}}(b)\big)
\]
for all $a,b\in\mathcal{A}$. The KMS-adjoint of $L_\phi$ is given by
$L^\flat_\phi=L_{\phi^\#\circ R}$, where $R$ denotes the unitary
antipode.

\end{theorem}
\begin{proof}
Let us observe first that a linear map $L\in\mathcal{L}(\mathcal{A})$ admits a
  KMS-adjoint if and only if it admits a GNS-adjoint, and that the two
  adjoints are related by
\begin{equation}\label{eq-flat-star}
L^\flat = \sigma_{\frac{i}{2}}\circ L^\star \circ \sigma_{-\frac{i}{2}}.
\end{equation}
Indeed, if the GNS-adjoint exists then, by \eqref{gns_symm_condition} and the
$\sigma$-invariance of $h$, we have
$$h\big( \sigma_{-\frac{i}{2}}(a)^*L(b)\big) = h\big(L^\star (\sigma_{-\frac{i}{2}}(a))^* b \big)
= h\big( (\sigma_{\frac{i}{2}} \circ L^\star \circ \sigma_{-\frac{i}{2}})(a)^* \sigma_{-\frac{i}{2}}(b) \big). $$
Comparing with Equation \eqref{kms_symm_condition_alg} and using the
faithfulness of the Haar state, we deduce that $L^\flat$ exists and satisfies
\eqref{eq-flat-star}. Conversely, if the KMS-adjoint exists then, using similar
arguments, we show that the GNS-adjoint exists and $L^\star = \sigma_{-\frac{i}{2}}
\circ L^\flat \circ \sigma_{\frac{i}{2}}$.

Now, it follows from Proposition \ref{prop-self-adj} that $L^\flat_\phi$ exists and for all $a\in\mathcal{A}$ we have
\begin{eqnarray*}
L_\phi^\flat(a) &=&
f_{-\frac{1}{2}}\star\big(L_{\phi^\#\circ S}(
f_{\frac{1}{2}}\star a \star f_{\frac{1}{2}})\big)\star f_{-\frac{1}{2}} \\
&=& f_{-\frac{1}{2}}\star (\phi^\#\circ S) \star f_{\frac{1}{2}}\star a \\
&=& a_{(1)} \big ( f_{-\frac{1}{2}}\star (\phi^\#\circ S) \star f_{\frac{1}{2}}\big ) (a_{(2)})\\
&=& a_{(1)} f_{-\frac{1}{2}}(a_{(2)}) (\phi^\#\circ S)(a_{(3)})
f_{\frac{1}{2}}(a_{(4)}) \\
&=&a_{(1)} \big((\phi^\#\circ S)(f_{\frac{1}{2}}\star a_{(2)} \star
f_{-\frac{1}{2}})\big) \\
&=& a_{(1)}(\phi^\#\circ R)(a_{(2)}) = L_{\phi^\#\circ R}(a),
\end{eqnarray*}
since $R(a) = S(f_{\frac{1}{2}}\star a \star
f_{-\frac{1}{2}})$, see Equation \eqref{eq-R}.
\end{proof}

\begin{corollary} \label{cor_symmetries}
 Suppose that $\phi\in \A'$. Then
\begin{enumerate}
 \item $L_\phi$ is GNS-symmetric if and only if $\phi$ satisfies $\phi^\# \circ S=\phi$.
 \item $L_\phi$ is KMS-symmetric if and only if $\phi$ satisfies $\phi^\#\circ R=\phi$.
\end{enumerate}
\end{corollary}

The generating functionals $\phi$ of L\'evy processes are necessarily hermitian
($\phi^\#=\phi$). We call a hermitian functional $\phi$ on $\A$ $\phi$
\emph{GNS-symmetric} if it is invariant under the
antipode: $\phi\circ S=\phi$, and \emph{KMS-symmetric} if it is invariant
under the unitary antipode: $\phi\circ R=\phi$.

\begin{remark} \label{remark_symmetry}
A hermitian $\phi$ is GNS-symmetric if and only if each matrix $\phi^{(s)}=[\phi(u_{jk}^{(s)})]_{j,k}$ is hermitian:
$$\phi(u_{jk}^{(s)})=(\phi\circ S)(u_{jk}^{(s)}) =\phi((u_{kj}^{(s)})^*) =\overline{\phi(u_{kj}^{(s)})}.$$
\end{remark}

 We shall show now that invariance under the phase in the polar
decomposition of the antipode has also an influence on the properties of
$L_\phi$.
\begin{proposition}\label{prop-modgr}
Let $\phi\in \mathcal{A}'$. Then the following conditions are equivalent:
\begin{enumerate}
 \item $L_\phi$ commutes with the modular automorphism group $\sigma$,
 \item $\phi$ commutes with the Woronowicz characters: $\phi\star f_{z} = f_{z}\star \phi$ for $z\in \C$,
 \item $\phi \circ \tau_{\frac{i}{2}} = \phi$.
\end{enumerate}
\end{proposition}
\begin{proof}
By Equation \eqref{eq-mod-gr-Wor-ch}, we have $L_\phi\circ \sigma_t=\sigma_t\circ L_\phi$ if and only if
\[
\phi \star f_{it}\star a \star f_{it} =f_{it} \star \phi \star a \star f_{it}
\]
for all $a\in\mathcal{A}$. Convolving by $f_{-it}$ from the right and applying
the counit, we see that $L_\phi$ commutes with the modular automorphism group, if and only
if
\[
\phi\star f_{it} = f_{it}\star \phi
\]
for all $t\in\mathbb{R}$, which is equivalent to
\begin{equation} \label{gen_comm_char}
\phi\star f_{z} = f_{z}\star \phi
\end{equation}
for all $z\in\mathbb{C}$ by uniqueness of analytic continuation. We have
shown this way that $(1) \Leftrightarrow (2)$.

{}From Equation \eqref{gen_comm_char} we deduce immediately that
$$\phi \circ \tau_{z}(a) = \phi(f_{iz} \star a \star f_{-iz}) = (f_{-iz} \star \phi \star f_{iz})(a) =\phi(a),$$
so $(2)$ implies $(3)$.

Finally, let us see that $(3)$ implies $(2)$. For that we adopt the matrix notation from \cite{woronowicz87}: $$F^{(s)}=[f_{-1}(u_{jk}^{(s)})]_{j,k=-s}^s \quad \mbox{and}\quad  \phi^{(s)}=[\phi(u_{jk}^{(s)})]_{j,k=-s}^s.$$

{}From therein we know that $F^{(s)}$ is invertible and positive and that
$f_z(u^{(s)})= (F^{(s)})^{-z}$. If $\phi \circ \tau_{\frac{i}{2}} = \phi$, then
by the definition of $\tau_z$ we have $\phi\star f_{-\frac12} =
f_{-\frac12}\star \phi$ and also $\phi\star f_{-1} = f_{-1}\star \phi$. This
means that
$$\phi^{(s)}F^{(s)}=F^{(s)} \phi^{(s)} $$
and by the functional calculus $\phi^{(s)}$ must commute with all $(F^{(s)})^z$ for $z\in \C$. This translates into $\phi\star f_{z} = f_{z}\star \phi$ for all $z\in \C$.
\end{proof}

It is known that on von Neumann algebras GNS-symmetry is a stronger condition than the KMS-one
(cf. \cite[Remarks after Definition 2.31]{cipriani08}). The previous
observation allows to provide a simple proof of this fact in our setting.
\begin{corollary} \label{gns_stronger_kms}
If $\phi$ is GNS-symmetric, then $\phi$ commutes with all Woronowicz characters and is KMS-symmetric.
\end{corollary}

\begin{proof}
 For GNS-symmetric $\phi$ we have $\phi=\phi \circ S^2=\phi \circ \tau_i$, which translates into $\phi\star f_{-1} = f_{-1}\star \phi$. From the proof of Proposition \ref{prop-modgr} we see that this implies that $\phi$ is invariant under all $\tau_z$ ($z\in \C)$ or, equivalently, commutes with all Woronowicz characters. In particular $\phi = \phi \circ \tau_{\frac{i}{2}}$ and
$$\phi = \phi \circ \tau_{\frac{i}{2}} = (\phi \circ S) \circ  \tau_{\frac{i}{2}} =\phi \circ R.$$
\end{proof}

\begin{remark} \label{rem_kms_gns}
 If the algebra $\A$ is of Kac type ($S^2=\id$), then $R=S$ and the notions of
GNS-symmetry and KMS-symmetry coincide. However, Example \ref{ex_KMS_not_sym}
shows that in general KMS-symmetry is a weaker condition than GNS-symmetry.
\end{remark}

We end this Section with an observation linking the symmetries of the generators
and the related Markov semigroups.
\begin{theorem} \label{thm_symmetries_on_Markov}
Let $(T_t)_{t\ge 0}$ be the Markov semigroup of a L\'evy process on
$\mathcal{A}$ with generating functional $\phi$.
\begin{itemize}
\item[($a$)] The following three conditions are equivalent:
\begin{itemize}
 \item[($a1$)] $\phi$ is KMS-symmetric.
 \item[($a2$)] $L_\phi$ is KMS-symmetric.
 \item[($a3$)] for each $t\geq 0$, $T_t$ is KMS-symmetric on $\mathsf{A}$ $($see Definition \ref{kms_symm_cstar}$)$.
\end{itemize}
\item[($b$)] The following four conditions are equivalent:
\begin{itemize}
 \item[($b1$)] $\phi$ is GNS-symmetric.
 \item[($b2$)] $L_\phi$ is GNS-symmetric.
 \item[($b2'$)] $L_\phi$ satisfies the \emph{quantum detailed balance} condition,
i.e.\ we have
\begin{equation}\label{eq-qdb}
h\big(a L_\phi(b)\big) = h\big(L_\phi(a)b\big) \quad \mbox{for} \quad a,b\in\mathcal{A}.
\end{equation}
 \item[($b3$)] $(T_t)_{t\ge 0}$ satisfies the \emph{quantum detailed balance} condition, i.e.
\eqref{eq-qdb} holds for all $T_t$, $t\ge 0$.
\end{itemize}
\end{itemize}
\end{theorem}

\begin{proof}
The equivalences $(x1)\Leftrightarrow (x2)$ follow from Corollary
\ref{cor_symmetries}.

 The KMS-symmetry as well as the GNS-symmetry of $\phi$ is preserved under the
convolution powers (for example, if $\phi(Sa)=\phi(a)$, then $(\phi\star
\phi)(Sa)= (\phi \otimes \phi)(S(a_{(2)})\otimes S(a_{(1)})) =
\phi(a_{(1)})\phi(a_{(2)})= (\phi\star \phi)(a).$) Since $L_\phi^n (a) =
\phi^{\star n} \star a$, we see that both kinds of symmetry are also preserved
for the powers of $L_\phi$.
This implies that for $(T_t)_{t\ge 0}$, being of the form $T_t=\exp t L_\phi$,
the KMS-symmetry or condition \eqref{eq-qdb} of $(T_t)_{t\ge 0}$ is equivalent
to KMS-symmetry or \eqref{eq-qdb} of $L_\phi$.

Finally we need to check that $(b2')\Leftrightarrow (b1)$. Assume that $L_\phi$ satisfies \eqref{eq-qdb}.
Then, by Proposition \ref{prop-self-adj}, $L_\phi$ satisfies
$$L_\phi (a)=L^\star (a^*)^* = (\phi^\star \star a^*)^* = a_{(1)} \phi \big( S(a_{(2)}^*)^* \big) =L_{\phi \circ S^{-1}}(a),$$ which implies $\phi\circ S=\phi$.

Conversely, if $\phi\circ S = \phi$, then by the same calculation we see that
$$h\big(L_\phi(a)b\big) =h\big(L_{\phi\circ S^{-1}}(a)b\big) =h\big(L^\star_\phi(a^*)^*b\big) = h\big(a L_\phi(b)\big).$$
\end{proof}

\section{Sch\"{u}rmann triples corresponding to KMS-symmetric generators}
\label{sec_schurmann}

 In this Section we give a method to produce KMS-symmetric generating
functionals. To this aim, we recall the notion of a Sch\"{u}rmann triple and
describe its behavior under the composition of an arbitrary generator with the
unitary antipode.

Our steps are motivated by the following easy observation.

\begin{proposition}
Let $\phi$ be a generating functional of a L\'evy process. Then $\phi+\phi\circ
R$ is a KMS-symmetric generating functional of a L\'evy process.
\end{proposition}

\begin{proof}
Since $R(a^*)=R(a)^*$ and $\e(R(a))=\e(a)$, we easily check that the Schoenberg
criteria for a generating functional are satisfied for $\phi+\phi\circ R$.
Moreover, $R^2={\rm id}$ implies that $\phi+\phi\circ R$ is invariant under the
unitary antipode.
\end{proof}
Note that the same procedure cannot be applied to the GNS-symmetric case, since $S$ does not preserve the positivity and is not involutive.

\medskip
For a pre-Hilbert space $D$ we denote by $\mathcal{L}^{\#}(D)$ the set of all
operators from $D$ to $D$ which admit an adjoint.
\begin{definition}
A \emph{Sch\"{u}rmann triple} on a $*$-bialgebra $\A$ with counit $\e$ is a
triple $((\pi,D), \eta, \phi)$ consisting of:
\begin{enumerate}
 \item a unital $*$-representation $\pi: \A \to \mathcal{L}^{\#}(D)$ of $\A$ on
some pre-Hilbert space $D$,
 \item a linear map $\eta: \A \to D$, called \emph{cocyle}, such that
$$\eta(ab)=\pi(a)\eta(b)+\eta(a)\e(b) \quad \mbox{for all}\; a,b \in \A ,$$
 \item a hermitian linear functional $\phi:\A \to \C$ satisfying
$$\phi(ab)=\langle \eta(a^*), \eta(b) \rangle \quad \mbox{for}\quad a,b\in \ker
\e.$$
\end{enumerate}
\end{definition}

Sch\"{u}rmann proved (cf. \cite{schuermann93}) that for any generating
functional $\phi$ of a L\'evy process there exists a Sch\"{u}rmann triple
$((\pi,D), \eta, \phi)$ (such that the generating functional is the last
ingredient of the triple). Moreover, the Sch\"{u}rmann triple is uniquely
determined (modulo unitary equivalence) provided that $\eta$ is surjective.

\begin{definition}
 Given a pre-Hilbert space $D$, the \emph{opposite space} $D^\op$ is defined as
$D^\op=\{\overline{v}: v\in D\}$ (the set of the same elements as $D$) with the
same addition $\overline{v}+\overline{w}=\overline{v+w}$, but with the scalar
multiplication given by
$\lambda \cdot \overline{v}= \overline{\overline{\lambda} v}$
and with the scalar product
$ \langle \bar{v}, \bar{w} \rangle_{\op} = \langle w, v \rangle.$

Given a unital $*$-representation $\pi: \A \to \mathcal{L}^{\#}(D)$ we define
$\pi^\op: \A \to \mathcal{L}^{\#}(D^\op)$ by the formula
$$\pi^\op(a)\bar{v} = \overline{(\pi\circ R)(a^*) v}, \quad \bar{v}\in D^\op.$$
We check directly that $\pi^\op$ is unital, multiplicative, and $*$-preserving,
so it is a $*$-representation of $\A$ on $D^\op$. We shall call it the
\emph{opposite representation}.

\end{definition}

\begin{theorem} \label{thm_schurmann_triples}
If $\phi$ is a generating functional of a L\'evy process with the
Sch\"{u}rmann triple $((\pi,D), \eta, \phi)$ on $\A$, then $\phi\circ R$ is a
generating functional of a L\'evy process with the Sch\"{u}rmann triple
$((\pi^\op,D^\op), \eta^\op, \phi\circ R)$ on $\A$ where $\pi^\op$ is the
opposite representation  with the representation space $D^\op$ and $\eta^\op: \A
\to D^\op$ is defined by $\eta^\op(a)=\overline{\eta(R(a^*))}$.
\end{theorem}

\begin{proof}
 Let $\phi$ be a generating functional of a L\'evy process. Then it follows from the
properties of $R$, mentioned after formula \eqref{decomposition_of_s}, that
$\phi\circ R$ is hermitian, conditionally positive and vanishes at $\1$. By the
Schoenberg correspondence, $\phi\circ R$ is a generating functional of a L\'evy process.

 Now we want to check that $((\pi^\op,D^\op), \eta^\op, \phi\circ R)$ is a
Sch\"{u}rmann triple. For that, note that $\eta^\op$ is linear and by the
cocycle property of $\eta$ we have
\begin{eqnarray*}
 \eta^\op (ab) &=& \overline{\eta(R((ab)^*))}=\overline{\eta(R(a^*)R(b^*))} \\
&=&\overline{\pi(R(a^*))\eta(R(b^*))}+ \overline{\eta(R(a^*))\e(R(b^*))} \\
&=&\pi^\op(a)\overline{\eta(R(b^*))}+ \overline{\eta(R(a^*))}\e(b)\\
&=& \pi^\op(a)\eta^\op (b)+ \eta^\op (a)\e(b).
\end{eqnarray*}
Moreover, $\phi\circ R$ is linear and hermitian, and for $a,b\in \ker \e$ we have
\begin{eqnarray*}
 \langle \eta^\op (a^*), \eta^\op(b) \rangle_\op
&=& \langle \overline{\eta(R(a))}, \overline{\eta(R(b^*))} \rangle_\op = \langle \eta(R(b^*)), \eta (R(a))\rangle \\
&=& \phi \big( R(b)R(a) \big) = (\phi \circ R)(ab).
\end{eqnarray*}
\end{proof}

\begin{corollary}
 If $\phi$ is invariant under $R$ and $((\pi,D), \eta, \phi)$ is the related
surjective Sch\"{u}rmann triple, then $\pi$ is equivalent to its opposite
representation $\pi^{\op}$.
\end{corollary}

\begin{corollary}\label{cor_schurmann}
If $\phi$ is a generating functional of a L\'evy process with surjective
Sch\"{u}rmann triple $((\pi,D), \eta, \phi)$ on $\A$, then $\big( (\pi\oplus
\pi^{\op},D\oplus D^\op), \eta\oplus\eta^\op, \phi+\phi\circ R \big)$ is a
Sch\"{u}rmann triple of a KMS symmetric generator $\phi+\phi\circ R$.
\end{corollary}
 Note that the Sch\"{u}rmann triple $\big( (\pi\oplus \pi^{\op},D\oplus
D^\op), \eta\oplus\eta^\op, \phi+\phi\circ R \big)$ in the Corollary
\ref{cor_schurmann} is not necessarily
surjective, even if the triple $((\pi,D), \eta, \phi)$ is surjective. This is
for example the case if $\phi$ is already KMS symmetric -- then the range of
$\eta\oplus\eta^\op$ is the diagonal of $D\oplus D^\op$.

\begin{remark} \label{rem_bounded_reps}
Let $\phi$ be a generating functional of a L\'evy process on $\A$ with the
associated Sch\"urmann triple $((\pi,D), \eta, \phi)$, i.e.\ $((\pi,D), \eta, \phi)$ is the unique Sch\"urmann triple for $\phi$ with a surjective cocycle.
If $\A$ is an algebraic quantum group ``of compact type``, i.e.\ is the $*$-subalgebra of polynomials of a compact quantum group $\G=(\mathsf{A},\Delta)$, then $\A$ is linearly spanned by the coefficients of unitary corepresentations and thus for every $a\in \A$, $\pi(a)$ is a bounded operator in $D$. In this
case the space $D$ can be completed to a Hilbert space $H$, $\eta:\A\to D$ turns into a cocycle $\eta:\A\to H$ with dense image, and $\pi$ maps $\A$ to $\mathcal{B}(H)$.
\end{remark}

\section{Generating functionals invariant under adjoint action}
\label{sec_adinv}

On classical Lie groups, central measures play an important role in harmonic
analysis and the study
of L\'evy processes. A measure $\mu$ on a topological group $G$ is called
\emph{central}, if it commutes with all other measures (w.r.t.\ to the
convolution). This is the case if
\[
\int_G f(gxg^{-1}){\rm d}\mu(x) = \int_G f(x){\rm d}\mu(x)
\]
for all $g\in G$ and $f\in C(G)$, or, equivalently, if $\delta_g\star\mu\star\delta_{g^{-1}}=\mu$
for all $g\in G$. On compact quantum groups we don't have Dirac measures, but
we can translate this condition to
\[
\psi_{(1)} \star \mu \star \hat{S}(\psi_{(2)}) = \psi(1) \mu
\]
for all functionals $\psi:\mathcal{A}\to\mathbb{C}$, for which $({\rm
  id}\otimes \hat{S})\circ\hat{\Delta} (\psi)= \psi_{(1)}\otimes
\hat{S}(\psi_{(2)})$ can be defined, i.e.\ for functionals which belong to the
algebra of smooth functions $\hat{\mathcal{A}}$ on the dual discrete quantum
group. This condition is equivalent to invariance of the functional $\mu$ under the
adjoint action, see below.

In this Section we will study ${\rm ad}$-invariance for functionals on compact
quantum groups. On cocommutative compact quantum groups (i.e.\ such that $\tau \circ \Delta = \Delta$,
where $\tau$ is the flip operator $\tau (x \otimes y) = y\otimes x$) all functionals are
${\rm ad}$-invariant, but on non-cocommutative compact quantum groups, ${\rm
  ad}$-invariance characterizes an interesting class of functionals that
share many similar properties with central measures. After reviewing several
characterizations and showing that the ${\rm ad}$-invariant functionals are
exactly those that belong to the center of $\mathcal{A}'$,  we show that it is possible to construct from a given functional an $\ad$-invariant one. But this construction does not preserve positivity.

\medskip
Recall that the \emph{adjoint action} of a Hopf algebra is defined by ${\rm
  ad}:\mathcal{A}\to\mathcal{A}\otimes\mathcal{A}$,
\[
{\rm ad}(a) = a_{(1)} S(a_{(3)}) \otimes a_{(2)}
\]
for $a\in\mathcal{A}$, see, e.g., \cite{majid95}, \cite[Section 1.3.4]{klimyk+schmuedgen97}.

The adjoint action is a left coaction, i.e.\ we have
\begin{eqnarray*}
({\rm id}\otimes {\rm ad})\circ {\rm ad} &=& (\Delta\otimes {\rm id})\circ{\rm
  ad}, \\
(\varepsilon\otimes {\rm id})\circ{\rm ad} &=& {\rm id}.
\end{eqnarray*}
But note that ${\rm ad}$ is not an algebra homomorphism.

\begin{definition}
We call a linear functional $\phi\in\mathcal{A}'$ \emph{${\rm ad}$-invariant}, if it satisfies
\[
({\rm id}\otimes \phi)\circ {\rm ad} = \phi \mathbf{1}_{\mathcal{A}}.
\]
Similarly, a linear map $L\in \mathcal{L}(\mathcal{A})$ is called \emph{${\rm
ad}$-invariant}, if it satisfies
\[
({\rm id}\otimes L)\circ {\rm ad} = {\rm ad}\circ L.
\]
\end{definition}

If the quantum group is cocommutative, then the
adjoint action is the trivial coaction $\ad (a) = \1 \otimes a$. Therefore in
this case all functionals are ${\rm ad}$-invariant.

It is straightforward to verify that the counit $\varepsilon$ and the Haar
state $h$ are ${\rm ad}$-invariant.

The following characterisations show that the ${\rm ad}$-invariant functionals
are a natural generalisation of central measures.

\begin{proposition}
Let $\phi\in \mathcal{A}'$. The following conditions are equivalent.
\begin{description}
\item[(a)]
$\phi$ is ${\rm ad}$-invariant.
\item[(b)]
We have
\[
\psi_{(1)} \star \phi \star \hat{S}(\psi_{(2)}) = \psi(1) \phi
\]
for all $\psi\in\hat{\mathcal{A}}$.
\item[(c)]
$\phi$ commutes with all elements of $\hat{\mathcal{A}}$: $\phi \star \psi = \psi \star \phi$ for all $\psi \in \hat{\A}$.
\item[(d)]
$\phi$ belongs to the center of $\mathcal{A}'$: $\phi \star \psi = \psi \star \phi$ for all $\psi \in \A'$.
\end{description}
\end{proposition}
\begin{proof}
\begin{description}
\item[(a)$\Leftrightarrow$(b)]
If $\phi:\mathcal{A}\to\mathbb{C}$ is ${\rm ad}$-invariant, then we have
\[
a_{(1)}S(a_{(3)}) \phi(a_{(2)}) = \phi(a)1.
\]
Applying the functional $\psi=h_b\in \hat{\mathcal{A}}$ with
$b\in\mathcal{A}$ to this, we get
\begin{eqnarray*}
\psi(1)\phi(a) &=& \psi\big( a_{(1)}S(a_{(3)})\big) \phi(a_{(2)}) =
\psi_{(1)}(a_{(1)}) \psi_{(2)}\big(S(a_{(3)})\big) \phi(a_{(2)}) \\
&=&\psi_{(1)}(a_{(1)}) \phi(a_{(2)})\hat{S}(\psi_{(2)})(a_{(3)}) =
\big(\psi_{(1)}\star \phi \star \hat{S}(\psi_{(2)})\big)(a)
\end{eqnarray*}
for all $a\in\mathcal{A}$ and all $\psi\in\hat{\mathcal{A}}$. The converse
follows, because by the faithfulness of the Haar state on $\mathcal{A}$ we
have
\[
\forall b \in \mathcal{A},\quad \psi(a)=h(ba)=0 \quad \Rightarrow \quad a=0.
\]
\item[(c)$\Rightarrow$(b)]
This follows directly from the antipode axiom,
\[
\psi_{(1)} \star \phi \star \hat{S}(\psi_{(2)}) = \psi_{(1)} \star
\hat{S}(\psi_{(2)})\star \phi = \hat{\varepsilon}(\psi)\hat{1}\star \phi =
\psi(1)\phi,
\]
where $\hat{1}=\e$ is the unit of $\A'$.
\item[(a)$\Rightarrow$(c)]
Suppose that $\phi$ is ${\rm ad}$-invariant and apply $\psi\circ m \circ ({\rm
  id}\otimes \phi\otimes {\rm id})\circ ({\rm ad}\otimes {\rm id})$ to
$\Delta(a)$, then this gives
\[
\psi\big(a_{(1)}S(a_{(3)}) a_{(4)}\big) \phi(a_{(2)})
\]
which is equal to $\psi(a_{(1)})\phi(a_{(2)})= (\psi\star \phi)(a)$ by the
antipode axiom. On the other hand, using the ${\rm ad}$-invariance of $\phi$,
the same expression becomes
\[
\psi(1 a_{(2)}) \phi(a_{(1)}) = (\phi\star\psi)(a).
\]
\item[(c)$\Leftrightarrow$(d)]
This follows, because $\mathcal{A}'$ embeds into the multiplier algebra
$\mathcal{M}(\hat{\mathcal{A}})$ of $\hat{\mathcal{A}}$, since
\[
\psi\star h_a = h_c,\qquad h_a\star \psi = h_d
\]
with $c=\psi(S(a_{(1)}))a_{(2)}$, $d=\psi(S^{-1}(a_{(2)}))a_{(1)}$.
\end{description}
\end{proof}

\begin{corollary} \label{cor_adinvariance_algebra}
The ${\rm ad}$-invariant functionals form a unital subalgebra of
$\mathcal{A}'$ with respect to the convolution.
\end{corollary}

The following formula shows that the coproduct
$\Delta:\mathcal{A}\to\mathcal{A}\otimes\mathcal{A}$ is ${\rm ad}$-invariant,
if we define the adjoint action of $\mathcal{A}\otimes\mathcal{A}$ by ${\rm
  ad}^\otimes=(m\otimes {\rm id}\otimes{\rm id}) \circ ({\rm id}\otimes \tau\otimes{\rm id})\circ ({\rm ad}\otimes{\rm ad})$.

\begin{lemma} \label{lemma_adjoint}
The adjoint action satisfies the relation
\[
(m\otimes {\rm id}\otimes{\rm id}) \circ ({\rm id}\otimes \tau\otimes{\rm id})\circ ({\rm ad}\otimes{\rm
  ad})\circ \Delta = ({\rm id}\otimes \Delta)\circ {\rm ad}.
\]
\end{lemma}
\begin{proof}
Using Sweedler notation, we get
\begin{eqnarray*}
\lefteqn{(m\otimes {\rm id}\otimes{\rm id}) \circ ({\rm id}\otimes \tau\otimes{\rm id})\circ ({\rm ad}\otimes{\rm ad})\circ \Delta(a) =}\\
&=& a_{(1)}S(a_{(3)})a_{(4)}S(a_{(6)})\otimes a_{(2)}\otimes a_{(5)} \\
&=& a_{(1)}\varepsilon(a_{(3)})\mathbf{1}_{\mathcal{A}}S(a_{(5)})\otimes
a_{(2)}\otimes a_{(4)}
\end{eqnarray*}
for $a\in\mathcal{A}$, where we used the antipode property \eqref{antipode_property}. After further
simplification, using the counit property  \eqref{counit_property}, we get
\[
=  a_{(1)}S(a_{(4)})\otimes
a_{(2)}\otimes a_{(3)} = ({\rm id}\otimes \Delta)\circ{\rm ad}(a).
\]
\end{proof}

\begin{lemma} \label{lemma_ad_invariance}
Let $\phi\in \mathcal{A}'$. Then $\phi$ is ${\rm ad}$-invariant if and only if
$L_\phi$ is ${\rm ad}$-invariant.
\end{lemma}
\begin{proof}
Let us observe that
\begin{eqnarray*}
\lefteqn{ ({\rm id}\otimes L_\phi)\circ {\rm ad} (a)
=(\id \otimes \id \otimes \phi ) (a_{(1)} S(a_{(4)}) \otimes a_{(2)}\otimes a_{(3)})} \\
&=& (\id \otimes \id \otimes \phi ) (a_{(1)} S(a_{(3)}) a_{(4)} S(a_{(6)})\otimes a_{(2)}\otimes a_{(5)})  \quad  \mbox{(cf. Lemma \ref{lemma_adjoint})}\\
&=& a_{(1)} S(a_{(3)}) \; a_{(4)} S(a_{(6)})\phi(a_{(5)})\otimes a_{(2)}.
\end{eqnarray*}
If we assume that $\phi$ is ad-invariant, then
$$ ({\rm id}\otimes L_\phi)\circ {\rm ad} (a) = a_{(1)} S(a_{(3)}) \phi(a_{(4)})\otimes a_{(2)}
= \phi(a_{(2)}) \ad( a_{(1)})= \ad \circ L_\phi (a).$$

On the other hand, if we suppose that $L_\phi$ is ad-invariant, then the application of $(\id \otimes \e)$ to both sides of the equation
$$\phi(a_{(4)})  a_{(1)} S(a_{(3)})\otimes a_{(2)}= a_{(1)} S(a_{(3)}) \, a_{(4)} S(a_{(6)})\phi(a_{(5)}) \otimes a_{(2)}$$
gives the ad-invariance of $\phi$.
\end{proof}

We can use the Haar state to produce $\ad$-invariant functionals.

\begin{proposition}\label{prop-adh}
Denote by $\ad_h\in \mathcal{L}(\mathcal{A})$ the linear map given by
$$\ad_h=(h\otimes{\rm id})\circ \ad.$$
Then $\phi_{\ad} :=\phi\circ \ad_h$ is
$\ad$-invariant for all $\phi\in\mathcal{A}'$.
\end{proposition}
\begin{proof}
Observe that by definition we have $\phi_{\ad}=\phi \circ \ad_h=(h\otimes \phi)\circ \ad$. Using the invariance of the Haar measure (Proposition \ref{prop-haar}) we check that
\begin{eqnarray*}
 \phi_{\ad} (a) \1 &=& h\big(a_{(1)} S(a_{(3)})\big) \phi(a_{(2)}) \1
= a_{(1)} S(a_{(5)}) h\big (a_{(2)} S(a_{(4)})\big) \phi( a_{(3)}) \\
&=& a_{(1)} S(a_{(3)}) \phi_{\ad} (a_{(2)}) = (\id \otimes \phi_{\ad}) \circ \ad(a).
\end{eqnarray*}
\end{proof}

Let us collect the basic properties of  $\ad_h$.
\begin{proposition} \label{ad_h_properties}
\begin{itemize}
\item[(a)]
$\ad_h\circ \ad_h=\ad_h$.
\item[(b)]
$(\phi\circ \ad_h)^\star = \phi^\star \circ \ad_h$ for all
$\phi\in\mathcal{A}'$.
\item[(c)]
A linear functional $\phi\in\mathcal{A}'$ is $\ad$-invariant if and only
if $\phi= \phi\circ\ad_h$.
\end{itemize}
\end{proposition}
\begin{proof}
Ad (a). Explicit calculations give
$$\ad_h\circ \ad_h (a)= h\big (h[ a_{(1)}S(a_{(5)})] a_{(2)}S(a_{(4)})\big ) a_{(3)}. $$
Apply the invariance of the Haar measure (Proposition \ref{prop-haar}) to the
element under the Haar state, and after the appropriate renumbering, we get
$$\ad_h\circ \ad_h (a)= h\big ( h(a_{(1)}S(a_{(3)}))\1 \big ) a_{(2)} = \ad_h (a). $$

Ad (b). Recall that $\phi^\star=\phi^{\#} \circ S$, where $S$ is the antipode and $\phi^{\#}(a)=\overline{\phi(a^*)}$. Then the assertion will follow if we show that
$$ [\ad_h \circ S(a)^*]^*= S\circ \ad_h (a).$$
Using the properties that $S\circ * \circ S \circ *=\id$ and $\Delta(S(a))=\tau \circ (S\otimes S) \circ \Delta (a)$ we check that
$\ad (S(a)^*) = S(a_{(3)})^* a_{(1)}^* \otimes S(a_{(2)})^*.$
Then since $h$ is hermitian, we have
\begin{eqnarray*}
[\ad_h \circ S(a)^*]^*&=& [(h\otimes \id) \circ (\ad \circ S)(a)^*]^*
= \overline{h\big (S(a_{(3)})^* a_{(1)}^* \big )} S(a_{(2)}) \\
&=& h\big (a_{(1)} S(a_{(3)}) \big ) S(a_{(2)}) = S (\ad_h (a)).
\end{eqnarray*}

Ad (c). First we check that for an ad-invariant functional $\phi$ we have $\phi=
\phi\circ{\rm ad}_h$:
$$\phi\circ{\rm ad}_h (a)= \phi \circ (h \otimes \id) \circ{\rm ad} (a)
= h \circ (\id \otimes \phi) \circ{\rm ad} (a)= h(\phi(a)\1) =\phi(a). $$
The converse follows immediately from Proposition \ref{prop-adh}.
\end{proof}

 Applying Lemma \ref{lemma_ad_invariance} and Corollary
\ref{cor_adinvariance_algebra}, we get an analogue of Theorem
\ref{thm_symmetries_on_Markov} for $\ad$-invariance.
\begin{corollary}
 Let $(T_t)_{t\ge 0}$ be the Markov semigroup of a L\'evy process on
$\mathcal{A}$ with generating functional $\phi$. The following three conditions are equivalent:
\begin{itemize}
 \item[($a1$)] $\phi$ is $\ad$-invariant.
 \item[($a2$)] $L_\phi$ is $\ad$-invariant.
 \item[($a3$)] for each $t\geq 0$, $T_t$ is $\ad$-invariant.
\end{itemize}
\end{corollary}

In the next proposition we show that $\ad$-invariance of functionals
can be characterized by the form of their characteristic matrices.

\begin{proposition} \label{prop_adinv_form}
A functional $\phi$ is \ad-invariant if and only if its characteristic matrices
$\big(\phi(u_{jk}^{(s)})\big)_{1\le j,k\le n_s}$ are multiples of the identity
matrix for all $s\in\mathcal{I}$, i.e.\ if there exist complex numbers $c_s$,
$s\in\mathcal{I}$, such that  $\phi(u_{jk}^{(s)})=c_s\delta_{jk}$ for all
$s\in\mathcal{I}$ and all $1\le j,k\le n_s$.
\end{proposition}

\begin{proof}
We use the orthogonality relation for the Haar measure \eqref{eq-pw} to show that for the \ad-invariant functional $\phi$ we have
\begin{eqnarray*}
 \phi (u_{jk}^{(s)})&=& \phi_\ad (u_{jk}^{(s)})
= \sum_{p,r=1}^n h(u_{jp}^{(s)} (u_{kr}^{(s)})^*) \phi(u_{pr}^{(s)}) = \frac{1}{D_s} \sum_{p,r=1}^n f_1(u_{rp}^{(s)}) \phi(u_{pr}^{(s)}) \cdot \delta_{jk},
\end{eqnarray*}
 and we observe that the constant $\frac{1}{D_s} \sum_{p,r=1}^n
f_1(u_{rp}^{(s)}) \phi(u_{pr}^{(s)})$ does not depend on $j$ or $k$.
Reciprocally, if $\phi$ is of this form, then we check that $\phi=\phi_\ad$ and,
by (c) in Proposition \ref{ad_h_properties}, $\phi$ is \ad-invariant.
\end{proof}

In general, the mapping ${\rm ad}_h^*:\phi \mapsto \phi_{\rm ad}$ in Proposition
\ref{prop-adh} preserves neither hermiticity nor positivity, see Example
\ref{ex_adinvariance_notherm}. But \cite[Lemma 4.1]{voigt2011} and \cite[Theorem
4.5]{bedos+murphy+tuset03} suggest that some properties of ${\rm ad}_h$ can be
improved if we replace the antipode by the twisted antipode defined by
$\widetilde{S}(a) =
f_1\star S(a)$ for $a\in\mathcal{A}$.

\begin{theorem}\label{thm-hatS}
Let $\mathbb{G}$ be a compact quantum group with dense $*$-Hopf algebra
$\mathcal{A}={\rm Pol}(\mathbb{G})$. Denote by $\widetilde{S}$ the twisted
antipode defined by $\widetilde{S}(a)=f_1\star S(a)=f_{-1}(a_{(1)}) S(a_{(2)})$
and denote by $\widetilde{\ad}$ the twisted adjoint action
$\widetilde{\ad}(a)=a_{(1)}\widetilde{S}(a_{(3)}) \otimes a_{(2)}$, $a\in\A$.
\begin{description}
\item[(a)]
The map $\widetilde{\rm ad}_h:\mathcal{A}\to \mathcal{A}$ defined by
\[
\widetilde{\rm ad}_h(a)=(h\otimes {\rm id})\circ \widetilde{\rm ad}(a) = h\big(a_{(1)}\widetilde{S}(a_{(3)})\big) a_{(2)}
\]
satisfies
\[
\widetilde{\rm ad}_h(a^*a) = (h\otimes {\rm id})\Big( \big(\widetilde{\rm ad}(a)\big)^*\, \widetilde{\rm ad}(a)\Big)
\]
for $a\in\mathcal{A}$ and therefore preserves positivity.
\item[(b)]
If $\mathbb{G}$ is of Kac-type, then we have
\[
\widetilde{\rm ad}_h\circ \widetilde{\rm ad}_h = \widetilde{\rm ad}_h.
\]
\end{description}
\end{theorem}
\begin{proof}
(b) was already shown in Proposition \ref{ad_h_properties}, since in the Kac case we have $\widetilde{\rm ad}_h={\rm ad}_h$.

Let us now prove (a). Since the twisted antipode is an algebra anti-homomor\-phism, we have
\begin{eqnarray*}
\widetilde{\rm ad}(a^*a) &=& a^*_{(1),j}a_{(1),k}\widetilde{S}(a_{(3),k}) \widetilde{S}(a_{(3),j}^*) \otimes a_{(2),j}^*a_{(2),k},
\end{eqnarray*}
where we put back summation indices to distinguish the sums coming from the first and the second factor.

Therefore
\begin{eqnarray*}
(h\otimes {\rm id})\left(\widetilde{\rm ad}(a^*a)\right) &=& h \left( a^*_{(1),j} a_{(1),k} \widetilde{S}(a_{(3),k})\widetilde{S}(a^*_{(3),j})\right) a^*_{(2),j} a_{(2),k} \\
&=& h \left( \sigma_i\left(\widetilde{S}(a^*_{(3),j})\right) a^*_{(1),j} a_{(1),k} \widetilde{S}(a_{(3),k})\right) a^*_{(2),j} a_{(2),k}.
\end{eqnarray*}
Now for any $b\in\mathcal{A}$,
\begin{eqnarray*}
\sigma_i(\widetilde{S}(b^*)) &=& f_{-1}\star\widetilde{S}(b^*)\star f_{-1}
=  S(b^*)\star f_{-1} = S^{-1}(b)^*\star f_{-1} \\
&=& (S^{-1}(b)\star f_1)^* = \big(f_1\star S(b)\big)^* = \widetilde{S}(b)^*,
\end{eqnarray*}
and so we get
\begin{eqnarray*}
(h\otimes {\rm id})\left(\widetilde{\rm ad}(a^*a)\right) &=& h \left(\big(a^*_{(1),j}\widetilde{S}(a_{(3),j})\big)^* a_{(1),k} \widetilde{S}(a_{(3),k})\right) a^*_{(2),j} a_{(2),k} \\
&=&(h\otimes {\rm id})\Big( \big(\widetilde{\rm ad}(a)\big)^* \widetilde{\rm ad}(a)\Big).
\end{eqnarray*}
\end{proof}

Recall that the linear span of the characters of the irreducible unitary
corepresentations of a compact quantum group is an algebra
\begin{equation} \label{eq_central_functions}
\mathcal{A}_0= {\rm span}\left\{ \chi_s = \sum_{j=1}^{n_s} u^{(s)}_{jj}:
  s\in\mathcal{I}\right\},
\end{equation}
called the \emph{algebra of central functions} on $\mathbb{G}$.

Note that $\widetilde{\rm ad}_h(\mathcal{A})\subseteq\mathcal{A}_0$. Indeed, $\widetilde{\rm ad}_h$ acts on coefficients of irreducible unitary
corepresentations as
\begin{eqnarray}
\widetilde{\rm ad}_h(u_{jk}^{(s)}) &=& \sum_{p,q=1}^{n_s} h\Big(u_{jp}^{(s)} \widetilde{S}(u^{(s)}_{qk})\Big)u^{(s)}_{pq} \nonumber \\
&=&\sum_{p,q,\ell=1}^{n_s} h\Big(u_{jp}^{(s)} S(u^{(s)}_{\ell k})\Big)f_{-1}(u^{(s)}_{q\ell})u^{(s)}_{pq} \nonumber \\
&=&\sum_{p,q,\ell=1}^{n_s} h\Big(u_{jp}^{(s)} \big(u^{(s)}_{k\ell}\big)^*\Big)f_{-1}(u^{(s)}_{q\ell})u^{(s)}_{pq} \nonumber \\
&=&\frac{1}{D_s} \sum_{p,q,\ell=1}^{n_s} \delta_{jk}f_1(u^{(s)}_{\ell p})f_{-1}(u^{(s)}_{q\ell})u^{(s)}_{pq} \nonumber \\
&=& \frac{\delta_{jk}}{D_s} \sum_{p=1}^{n_s} u^{(s)}_{pp}
\label{eq-adh}
\end{eqnarray}

We see that we can use $\widetilde{\rm ad}_h^*:\phi \mapsto \phi\circ\widetilde{\rm ad}_h $ to produce $\ad$-invariant functionals.

\begin{corollary}\label{cor-bijection}
The linear map $\widetilde{\rm ad}_h^*:(\mathcal{A}_0)'\to \mathcal{A}'$,
$\widetilde{\rm ad}_h^*(\phi)=\phi\circ\widetilde{\rm ad}_h$, maps functionals on
$\mathcal{A}_0$ to {\rm ad}-invariant functionals on $\mathcal{A}$. It maps
states on $\mathcal{A}_0$ to states on $\mathcal{A}$.

If $\mathbb{G}$ is of Kac type, then $\widetilde{\rm ad}_h^*$ defines bijections
between states on $\mathcal{A}_0$ and {\rm ad}-invariant states on $\mathcal{A}$,
and between generating functionals on $\mathcal{A}_0$ and
{\rm ad}-invariant generating functionals on $\mathcal{A}$.
\end{corollary}
\begin{proof}
 It follows immediately from Equation \eqref{eq-adh} and Proposition
\ref{prop_adinv_form} that for any $\phi\in\mathcal{A}'$ the functional
$\widetilde{\rm ad}_h^*(\phi)=\phi\circ\widetilde{\rm ad}_h$ is ad-invariant,
since
\[
\phi\circ\widetilde{\rm ad}_h(u_{jk}^{(s)}) = \frac{\delta_{jk}}{D_s} \phi\left(\sum_{\ell=1}^{n_s} u^{(s)}_{\ell \ell}\right).
\]
 By Theorem \ref{thm-hatS}, $\widetilde{\rm ad}_h^*$ maps positive functionals
to positive functionals. Since we have also $\widetilde{\rm ad}(\1)=\1$ and
\[
\phi\circ \widetilde{\rm ad}_h(\1)=\phi(\1),
\]
it follows that $\widetilde{\rm ad}_h^*$ maps states on $\mathcal{A}$ onto ad-invariant states on $\mathcal{A}$.

 In the Kac case we have furthermore $\varepsilon\circ \widetilde{\rm
ad}_h=\varepsilon\circ{\rm ad}_h=\varepsilon$, so in this case $\widetilde{\rm
ad}_h$ maps the kernel of the counit onto itself and therefore $\widetilde{\rm
ad}_h^*$ maps generating functionals on $\mathcal{A}$ onto ad-invariant
generating functionals on $\mathcal{A}$.

 Conversely, any state or generating functional $\psi$ on $\mathcal{A}_0$ can be
extended to a state or generating functional $\hat{\psi}=\psi\circ\widetilde{\rm
ad}_h$ on $\mathcal{A}$. By Proposition \ref{prop_adinv_form} it is clear that
$\hat{\psi}$ is the unique ad-invariant extension of $\psi$.
\end{proof}

 In Section \ref{sec-On}, we will show that this Corollary allows to completely
characterize the ad-invariant generating functionals on the free orthogonal
quantum group $O_n^+$.

\section{Dirichlet forms}
\label{sec_dirichlet}

In this Section we determine explicitly the structure of the Dirichlet forms associated to KMS-symmetric generating functionals on compact quantum groups. In case of GNS symmetry, we also characterize the invariance under translation of generators on the algebra $\A$, in terms of an associated quadratic form on $\A$.

Recall that $L^2 (\mathsf{A},h)$ denotes the GNS Hilbert space of $(\mathsf{A},h)$ and that the cyclic vector $\xi_h =1_\mathsf{A}\in L^2 (\mathsf{A},h)$ represents the Haar state as $h(a)=\langle \xi_h ,a\xi_h\rangle$. From now on and until the end of Section \ref{sec_derivation}, we assume that the Haar state is faithful on the C$^*$-algebra $\mathsf{A}$ so that we can identify $\mathsf{A}$ with an involutive subalgebra
of the von Neumann algebra $L^\infty (\mathsf{A},h)$ of bounded operators on  $L^2 (\mathsf{A},h)$, generated by $\mathsf{A}$ by the GNS representation. As a consequence, the vector $\xi$ is cyclic
for the von Neumann algebra $L^\infty (\mathsf{A},h)$ too.

Notice that, as the Haar state $h$ is a $(\sigma ,-1)$-KMS state for the modular automorphism group $\sigma$ (see Section \ref{woronowicz_characters}), it follows, by the KMS theory (in particular \cite[Corollary 5.3.9]{bratteli+robinson97}), that the vector $\xi$ is also separating for the von Neumann algebra $L^\infty(\mathsf{A},h)$. This fact allows to apply the Tomita-Takesaki modular theory to the Haar state $h$ on the C$^*$-algebra $\mathsf{A}$ of the compact quantum group.

Recall also that the symmetric embedding is defined by
\[
i_s:\mathsf{A}\to L^2 (\mathsf{A},h), \quad i_s(a) = \Delta^{\frac{1}{4}}a\xi_h,
\]
where $\Delta$ denotes (exceptionally) the Tomita-Takesaki modular operator. This definition agrees with the one from Section \ref{sec_gns_kms} (page \pageref{symmetric_embedding}), since for $a\in \A$ we have
\[
i_s(a) = \Delta^{\frac{1}{4}}a\Delta^{-\frac{1}{4}}\xi_h = \sigma_{-\frac{i}{4}} (a)\xi_h .
\]

For a given KMS-symmetric generating functional $\phi$ of a L\'evy process on $\A$, and the related convolution operator $L_\phi (a) = \phi \star a$, we define the sesquilinear and the quadratic forms
\begin{eqnarray*}
\E_\phi \big(i_s(a),i_s(b) \big) &=& \big\langle i_s(a),i_s(-L_\phi (b)) \big\rangle
= -h \big( \sigma_{-\frac{i}{4}} (a)^* (\sigma_{-\frac{i}{4}} \circ L_\phi)(b)\big), \\
\E_\phi[i_s(a)]&=& \E_\phi \big( i_s(a),i_s(a) \big),
\end{eqnarray*}
on the domain
\[
D(\E_\phi)=\{i_s(a)\in L^2(\mathsf{A},h): a\in D(L_\phi) \mbox{ and } \E_\phi [i_s(a)]<\infty\}\, .
\]

The explicit values of the sesquilinear form $\E_\phi$ on the basis of the coefficients of the unitary corepresentations are the following
\begin{eqnarray*}
\E_\phi \big(i_s(u^{(s)}_{jk}),i_s(u^{(t)}_{lm}) \big) &=& \big\langle i_s(u^{(s)}_{jk}),i_s(-L_\phi (u^{(t)}_{lm})) \big\rangle
=\sum_r \big\langle i_s(u^{(s)}_{jk}),i_s(u^{(t)}_{lr})\big\rangle\phi (u^{(t)}_{rm})\\
&=& \sum_r h\big( \sigma_{-\frac{i}{4}} (u^{(s)}_{jk})^* \sigma_{-\frac{i}{4}} (u^{(t)}_{lr}) \big) \phi (u^{(t)}_{rm})\\
&=& \sum_{r,p,p'} h\big( (u^{(s)}_{jk})^* u^{(t)}_{pp'} \big) f_{\frac12}(u^{(t)}_{lp}) f_{\frac12}(u^{(t)}_{p'r}) \phi (u^{(t)}_{rm})\\
&=&\frac{\delta_{st}}{D_s} \sum_{p} f_{-1}(u^{(t)}_{pj})  \big) f_{\frac12}(u^{(t)}_{lp}) \sum_{r} f_{\frac12}(u^{(t)}_{kr}) \phi (u^{(t)}_{rm}) \\
&=&\frac{\delta_{st}}{D_s} f_{-\frac12}(u^{(t)}_{lj}) (f_{\frac12}\star \phi)(u^{(t)}_{km}).
\end{eqnarray*}

Since $\sigma_z$ leaves the subspaces $V_s$ invariant, hence
\[
i_s (V_s)=\sigma_{-\frac{i}{4}}(V_s)\xi_h =V_s\xi_h = i_h (V_s)\,
\]
and the operator defined by
\[
H_\phi i_s(a):=i_s(-L_\phi a),\qquad a\in D(H_\phi):= i_s(\A) \subset L^2(\mathsf{A},h)
\]
leaves invariant the subspaces
\[
E_s=V_s \xi_h = {\rm Span}\{u^{(s)}_{jk}\xi_h :j,k=1, \cdots ,n_s\}\subset L^2 (\mathsf{A},h),\quad s\in\mathcal{I}\, .
\]
Therefore, since $L^2 (A,h)=\bigoplus_{s\in\mathcal{I}} E_s$, the operator $H_\phi$ decomposes as
\[
H_\phi =\bigoplus_{s\in\mathcal{I}} H_\phi^s
\]
a direct sum of its restrictions $H_\phi^s$ on each finite dimensional subspace $E_s$.

\begin{theorem}
Let $\phi$ be a KMS-symmetric generating functional of a L\'evy process on $\A$. Then the operator $H_\phi$ is essentially self-adjoint, the quadratic form $\E_\phi$ is closable and its closure is a Dirichlet form.
\end{theorem}

\begin{proof}
The operator $H_\phi$ is a direct sum of bounded operators and is symmetric as $L_\phi$ is KMS symmetric. It follows that $H_\phi$ is essentially self-adjoint and its closure is given by
\[
D(\overline{H_\phi})=\{\xi =\oplus_{s\in \mathcal{I}}\, \xi_s\in L^2 (A,h): \sum_{s\in \mathcal{I}} \|H_\phi\xi_s\|^2 <+\infty\}\, .
\]
\[
\overline{H_\phi}(\oplus_{s\in \mathcal{I}}\, \xi_s)=\oplus_{s\in \mathcal{I}}\, H_\phi\xi_s\, ,\qquad \oplus_{s\in \mathcal{I}}\, \xi_s\in D(\overline{H_\phi})\, .
\]
As, by definition, $\E_\phi [\xi]=\langle\xi ,H_\phi\xi\rangle$ for $\xi\in D(H_\phi)$, we have that $\E_\phi$ is closable and its closure is given by
\[
D(\overline{\E_\phi})=\{\xi =\oplus_{s\in \mathcal{I}}\, \xi_s\in L^2 (A,h): \sum_{s\in \mathcal{I}} \langle \xi_s , H_\phi\xi_s\rangle <+\infty\}\, .
\]
\[
\overline{\E_\phi}[\oplus_{s\in \mathcal{I}}\, \xi_s]= \sum_{s\in \mathcal{I}} \langle \xi_s , H_\phi\xi_s\rangle\, ,\qquad \oplus_{s\in \mathcal{I}}\, \xi_s\in D(\overline{\E_\phi})\, .
\]

Now, the quantum Markov semigroup $T_t$ on the C$^*$-algebra $\mathsf{A}$, generated by $L_\phi$, is KMS symmetric, i.e. is $(\sigma ,-1)$-KMS symmetric in the sense of Definition 2.1 in\cite{cipriani98} (see also Definition 2.31 in \cite{cipriani08}). By Theorem 2.3 and Theorem 2.4 in \cite{cipriani98} (see also Theorem 2.39 and Theorem 2.44 in \cite{cipriani08}) the semigroup $e^{-t{\overline H_\phi}}$ on $L^2 (\mathsf{A},h)$ is Markovian so that the quadratic form $\overline{\E_\phi}$ is a Dirichlet form by Theorem 4.11 in \cite{cipriani97} (see also Theorem 2.52 in \cite{cipriani08}).
\end{proof}

\begin{remark} \label{rem_dirichlet_gns}
Using the embedding $i_h :\mathsf{A}\to L^2 (\mathsf{A},h)$, we can identify the Dirichlet form on $L^2 (\mathsf{A},h)$, associated to a KMS-symmetric generating
functional $\phi$, with the following quadratic form on the C$^*$-algebra $\mathsf{A}$
\[
\Q_\phi[a]=\E_\phi[i_h(a)]= - h \big( a^* (\sigma_{-\frac{i}{4}} \circ L_\phi
\circ \sigma_{\frac{i}{4}} )(b)\big)
\]
defined on ${\it dom}\, (\Q_\phi):=\{a\in A: i_h(a)\in {\it dom}\, (\E_\phi)\}$.
If furthermore, $\phi$ is GNS-symmetric, then $L_\phi$ commutes with the modular group $(\sigma_z)_z$ (Prop. \ref{prop-modgr} and Cor. \ref{gns_stronger_kms}) and one has
\[
\Q_\phi[a]= - h \big( a^* L_\phi (a)\big).
\]
\end{remark}

The next theorem shows that the Dirichlet forms
associated to GNS-symmetric L\'evy processes admit an additional invariance.

\begin{theorem}
Let $L$ be a GNS symmetric operator on $\A\subset L^2(\mathsf{A},h)$. Then the following conditions are equivalent:
\begin{enumerate}
\item There exists a functional $\phi\in\mathcal{A}'$ such that $L=L_\phi$, where $L_\phi=(\id\otimes \phi)\circ\Delta$;
\item $L$ is translation invariant on $\mathcal{A}$;
\item The semigroup $(T_t)_{t\ge 0}$ on $\mathcal{A}$ (or $\mathsf{A}_r$ or
$\mathsf{A}_u$) associated to $L$ by the formula $T_t|_{\mathcal{A}} =
\exp_{\star} tL$ is translation invariant.
\item the sesquilinear form $\mathcal{\Q}$ defined by $\Q (a,b) = -h(a^*L(b))$ on $\mathcal{A}$ satisfies
\begin{equation} \label{eq_dirichlet_inv}
 \Q (a,b){\bf 1} = (m_* \otimes \Q) (\Delta(a),\Delta(b)), \quad a,b \in \mathcal{A},
\end{equation}
where $m_*$ denotes the sesquilinear map obtained from the multiplication, namely,
$m_*(a, b)= a^*b$.
\end{enumerate}
\end{theorem}

\begin{proof}
We already observed in Subsection \ref{processes} the equivalence $(1) \Leftrightarrow (2)$ whereas the equivalence $(2)\Leftrightarrow (3)$ follows from Theorem \ref{prop-trans-inv}, so that we need to prove only $(1) \Leftrightarrow (4)$. Let us assume that $L$ satisfies
$$ (\id \otimes L) \circ \Delta =\Delta \circ L.$$
Then, using Sweedler notation and the invariance of the Haar state,
\begin{eqnarray*}
\lefteqn{(m_* \otimes \Q) (\Delta(a),\Delta(b)) = a_{(1)}^* b_{(1)} \Q (a_{(2)},b_{(2)})}\\
&=& - a_{(1)}^* b_{(1)} h \big( a_{(2)}^*L(b_{(2)}) \big)
= - (\id \otimes h) \big( (a_{(1)}^*\otimes a_{(2)}^*) (\id \otimes L) \Delta(b) \big)\\
&=& - (\id \otimes h) \big( (a_{(1)}^*\otimes a_{(2)}^*) (L(b)_{(1)} \otimes L(b)_{(2)}) \big)\\
&=& - (\id \otimes h) \Delta \big( a^* L(b) \big) = - h (a^*L(b)) {\bf 1} = \Q(a,b) {\bf 1}.
\end{eqnarray*}

On the other hand, if we assume that Equation \ref{eq_dirichlet_inv} holds, then
\begin{eqnarray*}
\lefteqn{(h\otimes h) (a^*\otimes {\bf 1})\Delta (b^*) \big ( (\id \otimes L) \Delta(c) \big)}\\
&=& (h \otimes h) \big( a^*b_{(1)}^*c_{(1)} \otimes b^*_{(2)}L(c_{(2)}) \big)
= -h \big( a^*b_{(1)}^*c_{(1)}\big) \Q \big( b_{(2)},c_{(2)} \big)\\
&=& -h \big( a^* (m_* \otimes \Q) (\Delta(b),\Delta(c)) \big) = -h(a^*) \Q (b,c)
\end{eqnarray*}
and
\begin{eqnarray*}
\lefteqn{(h\otimes h) (a^*\otimes {\bf 1})\Delta (b^*) (\Delta \circ L) (c) }\\
&=& (h \otimes h) \big( a^*b_{(1)}^*(Lc)_{(1)} \otimes b^*_{(2)}(Lc)_{(2)} \big)
= h \big( a^* (\id \otimes h) \Delta(b^*L(c))\big)\\
&=& h ( a^*) h(b^*L(c))  = - h(a^*) \Q (b,c).
\end{eqnarray*}
Since $\A\odot\A$ is the linear span of $(\A \otimes 1 ) \Delta(\A)$ and $h\otimes h$ is faithful on $\A\odot\A$, we conclude that $L$ is translation invariant.
\end{proof}

\begin{corollary}
Let $\phi\in\A'$ be the generating functional of a GNS-symmetric L\'evy process and $\E_\phi$ the associated Dirichlet form. Then the sesquilinear form $\Q$ on $\A$ defined by
\[
\Q (a,b):=\E_\phi (i_h (a), i_h (b))\qquad a,b\in\A
\]
satisfies Eq. $(\ref{eq_dirichlet_inv})$. Conversely, let $L$ is a GNS-symmetric operator on $\A$ such that $L(\1)=0$, $L$ is hermitian and positive on ${\rm ker}\,\varepsilon$, and the sesquilinear form on $\A$ defined by
\[
\Q (a,b):=-h(a^* Lb)\qquad a,b\in\A
\]
satisfies Eq. $(\ref{eq_dirichlet_inv})$. Then $L=L_\phi$ for a generating functional $\phi$ of a GNS-symmetric L\'evy process.
\end{corollary}

\section{Derivations, cocycles and Spectral Triples}
\label{sec_derivation}
In this section we associate to any L\'evy process on a CQG $\mathbb{G}=(\mathsf{A},\Delta)$, a natural derivation on its Hopf $^*$-subalgebra $\A$, with values in a Hilbert bimodule over the C$^*$-algebra $\mathsf{A}=C(\mathbb{G})$.
This gives rise, on the same bimodule, to a self-adjoint operator $D$, with respect to which we prove that the elements of $\A$ are "Lipschitz" in a natural, suitable sense. The construction makes essential use of the Sch\"{u}rmann triple associated to the generator of the process.

In case the GNS symmetry holds true, we will show that the derivation is,
essentially, a differential square root of the generator $H_\phi$. Moreover, if
the spectrum of $H_\phi$ on $L^2 (A,h)$ is discrete, then the Hilbert bimodule
and the operator $D$ form a spectral triple in the sense of the noncommutative
geometry of A. Connes \cite{connes94}. This fact suggests to refer to $D$ as the
{\it Dirac operator} associated to the process.

We remark that the role of GNS symmetry of the process is to provide a suitable closability property of the derivation, needed to prove that the Dirac operator $D$ is self-adjoint and that the spectrum of the {\it Dirac Laplacian} $D^2$ coincides with that of the generator $H_\phi$, away from zero.

We will show in Section 9 that in case the CQG is a compact Lie group and the L\'evy process is the Brownian motion associated to a given Riemannian metric, the differential calculus illustrated above reduces to the familiar one: the derivation coincides with the gradient operator and the Lipschitz property has the usual meaning.

\medskip
Consider on the Hopf $^*$-subalgebra $\A$ of a compact quantum group $\mathbb{G}=(\mathsf{A},\Delta)$, the generating functional $\phi\in\A^\prime$ of a L\'evy process and its associated Sch\"{u}rmann triple $((\pi, H_\pi ), \eta, \phi)$ on a Hilbert space $H_\pi$ (see Remark \ref{rem_bounded_reps}).

Denote by $\lambda_{L} , \lambda_{R} :\mathsf{A}\rightarrow B(L^2 (\mathsf{A},h))$ the left and right actions of $\mathsf{A}$ on the Hilbert space $L^2 (A,h)$
\[
\begin{split}
&\lambda_L (a)(b\xi_h):=ab\xi_h ,\\
&\lambda_R (a)(b\xi_h ):=ba\xi_h,\qquad a,b\in \mathsf{A},
\end{split}
\]
where $\xi_h\in L^2 (\mathsf{A},h)$ denotes the cyclic vector representing the Haar state. Recall now that $\Delta :\mathsf{A}\rightarrow \mathsf{A}\otimes \mathsf{A}$ is a morphism of C$^*$-algebras and $\lambda_L$, $\pi$ are representations of the C$^*$-algebra $\mathsf{A}$, so that $\lambda_L\otimes \pi$ is a representation of the C$^*$-algebra $\mathsf{A}\otimes \mathsf{A}$.
Correspondingly, consider the left and right actions $\rho_{L} ,\rho_R :\mathsf{A}\rightarrow B(L^2 (\mathsf{A},h)\otimes H_\pi)$ of $\mathsf{A}$ on the Hilbert space $L^2 (\mathsf{A},h)\otimes H_\pi$ defined by
\[
\begin{split}
&\rho_L :=(\lambda_L\otimes \pi )\circ\Delta\\
&\rho_R :=\lambda_R\otimes {\rm id}_{H_\pi}
\end{split}
\]
or, more explicitly, by
\begin{eqnarray*}
\rho_L (a)(b\xi_h\otimes v)&=& ((\lambda_L \otimes \pi)\circ\Delta (a) ) (b\xi_h\otimes v) = \sum a_{(1)}b\xi_h\otimes \pi(a_{(2)})v \\
\rho_R (a)(b\xi_h\otimes v)&=& \lambda_R (a)(b\xi_h )\otimes v=ba\xi_h \otimes v\, ,
\end{eqnarray*}
for $a,b\in \mathsf{A}$ and $v\in H_\pi$. The actions $\lambda_L\, ,\rho_L$ are continuous and form representations of the C$^*$-algebra $A$. Likewise, also the actions $\lambda_R\, ,\rho_R$ are continuous and form antirepresentations of the C$^*$-algebra $A$ or representations of the opposite C$^*$-algebra $A^{\rm op}$. Moreover, as $\lambda_L\, ,\lambda_R$ (resp. $\rho_L\, ,\rho_R$) commute, they provide a $A$-bimodule structure on the Hilbert space $L^2 (A,h)$ (resp. $L^2 (\mathsf{A},h)\otimes H_\pi$).
\vskip0.2truecm\noindent
In the following we shall adopt the simplified notations: for $a\in \mathsf{A}$ and $\xi\in L^2 (\mathsf{A},h)\otimes H_\pi$ we write
\begin{eqnarray*}
a\cdot \xi &:=\rho_L (a)\xi \, , \\
\xi\cdot a &:=\rho_R (a)\xi \, .
\end{eqnarray*}
Recall that we denote by $i_h :\mathsf{A}\rightarrow L^2 (\mathsf{A},h)$ the GNS embedding (cf. page \pageref{gns_embedding})
\[
i_h (a):=a\xi_h\qquad a\in \mathsf{A}.
\]

\begin{proposition}
Consider on the Hopf $^*$-subalgebra $\A$ of a compact quantum group $\mathbb{G}=(\mathsf{A},\Delta)$, the generating functional $\phi\in\A^\prime$ of a L\'evy process, its associated Sch\"{u}rmann triple $((\pi, H_\pi ), \eta, \phi)$ and the induced $\mathsf{A}$-bimodule structure on $L^2 (\mathsf{A},h)\otimes H_\pi$. Then the linear map defined by
\[
\partial :\A\rightarrow L^2 (\mathsf{A},h)\otimes H_\pi\qquad \partial :=(i_h\otimes \eta )\circ\Delta
\]
or, more explicitly, by
\[
\partial a = (i_h\otimes \eta )(\Delta a)=\sum a_{(1)}\xi_h \otimes \eta(a_{(2)})\qquad a\in\A\, ,
\]
is a derivation in the sense that it satisfies the Leibniz rule
\[
\partial (ab)=(\partial a)\cdot b + a\cdot (\partial b)\qquad a,b\in\A\, .
\]
\end{proposition}
\begin{proof}
The map is well defined because $\Delta (\A)\subseteq \A\otimes\A$ (where, forcing notation a little bit, we denoted by $\A\otimes\A$ the image in $\mathsf{A}\otimes \mathsf{A}$ of the subspace $\A\odot\A\subseteq \mathsf{A}\odot \mathsf{A}$ under the canonical quotient map from $\mathsf{A}\odot \mathsf{A}$ to $\mathsf{A}\otimes \mathsf{A}$).
\par\noindent
As the map $\eta :\A\rightarrow H_\pi$ is a 1-cocycle and the counit $\e :\A\rightarrow \mathbb{C}$ satisfies the identity $({\rm id}\otimes\e)\circ\Delta = {\rm id}$, i.e. $\sum b_{(1)}\e (b_{(2)})=b$, we have
\begin{eqnarray*}
\partial (ab) &=& (i_h \otimes \eta)(\Delta(ab))= (i_h \otimes \eta)( \Delta(a)\Delta(b))\\
&=& \sum_{j,k} a_{(1),j}\, b_{(1),k}\xi_h \otimes \eta(a_{(2),j}\,b_{(2),k}) \\
&=& \sum_{j,k} a_{(1),j}\, b_{(1),k}\xi_h \otimes [\pi(a_{(2),j})\eta(b_{(2),k})+\eta(a_{(2),j})\e(b_{(2),k})] \\
&=& \left( \sum_j \lambda_L (a_{(1),j}) \,\otimes \pi(a_{(2),j})\right)  \left(\sum_k b_{(1),k}\xi_h \otimes \eta(b_{(2),k})\right) \\
& & + \left( \sum_k \lambda_R (b_{(1),k})\otimes \e(b_{(2),k}){\rm id}_{H_\pi}\right)\left( \sum_j a_{(1)_j}\xi_h \otimes \eta(a_{(2),j})\right)  \\
&=& \rho_L (a)(\partial(b)) +  \left( \lambda_R ( \sum_k b_{(1),k}\e(b_{(2),k}))\otimes {\rm id}_{H_\pi}\right)(\partial(a)) \\
&=& \rho_L (a)(\partial(b)) +  \left( \lambda_R ( b)\otimes {\rm id}_{H_\pi}\right)(\partial(a)) \\
&=& \rho_L (a)(\partial(b)) +  \rho_R (b)(\partial(a))
a\cdot\partial(b) +  \partial(a)\cdot b\, .
\end{eqnarray*}
\end{proof}

\begin{proposition} \label{prop_derivation}
Let $\phi\in\A^\prime$ be a GNS-symmetric generating functional and consider the hermitian convolution generator $L_{\phi} :\A\rightarrow\A$, its Hilbert-space extension $(H_\phi ,D(H_\phi))$ as well as the Dirichlet form $(\E_\phi ,D(\E_\phi))$ (see Section \ref{sec_dirichlet}).
\par\noindent
Then the operator $d:D(d)\rightarrow L^2 (A,h)\otimes H_\pi$ defined as
\[
D(d):=i_h(\A)=\A\xi_h\subset L^2 (A,h)\, ,\qquad d(i_h(a)):=\partial a,\qquad a\in\A,
\]
is closable and
\[
\|d (i_h(a))\|^2_{L^2 (A,h)\otimes H_\pi} = 2\langle i_h(a),H_\phi i_h(a)\rangle_{L^2 (A,h)} = 2\, \E_\phi [i_h(a)],\quad a\in\A\, .
\]
\end{proposition}
\begin{proof}
We have
\begin{eqnarray*}
\|d (i_h(a))\|^2 &=& \langle \sum_j a_{(1),j} \otimes \eta(a_{(2),j}), \sum_k a_{(1),k} \otimes \eta(a_{(2),k}) \rangle \\
&=& \sum_{j,k} h \big( a_{(1),j}^* a_{(1),k}\big)
 \big[\phi(a_{(2),j}^*a_{(2),k}) - \e(a_{(2),j}^*)\phi(a_{(2),k}) - \phi(a_{(2),j}^*)\e(a_{(2),k}) \big] \\
&=& \sum_{j,k} h\big( a_{(1),j}^* a_{(1),k}\big)\phi(a_{(2),j}^*a_{(2),k})
-\sum_{j,k} h\big( a_{(1),j}^* a_{(1),k}\big)\e(a_{(2),j}^*)\phi(a_{(2),k}) \\
	& & - \sum_{j,k} h\big(  a_{(1),j}^* a_{(1),k} \big) \phi(a_{(2),j}^*)\e(a_{(2),k}).
\end{eqnarray*}
The first term vanishes because, by the GNS symmetry of $L_\phi$, we have
\[
\sum_{j,k} h\big( a_{(1),j}^* a_{(1),k}\big)\phi(a_{(2),j}^*a_{(2),k}) = h\big(\1 \cdot L_\phi (a^*a)\big) = h( L_\phi(\1)\, a^*a ) = 0.
\]
The second term becomes
\begin{eqnarray*}
\lefteqn{
\sum_{j,k} h\big( a_{(1),j}^* a_{(1),k}\big)\e(a_{(2),j}^*)\phi(a_{(2),k})}\\
&=& h\big( ( \sum_{j} a_{(1),j}\e(a_{(2),j}))^* ( \sum_{k} a_{(1),k}\phi(a_{(2),k}))\big)
=h\big( a^* L_\phi (a) \big)\, .
\end{eqnarray*}
The third term is the complex conjugate of the second term and, since it is real, they are the same. Hence the identity $\|d\xi\|^2 =2\langle \xi,H_\phi\xi\rangle_{L^2 (A,h)} = 2\E_\phi [\xi]$ holds for $\xi\in D(d)$. Since $\phi$ is also KMS symmetric, Lemma 7.1 implies that the quadratic form is closable. It follows that the operator $d$ is closable, too.
\end{proof}

{\it From now on we will denote by the same symbol $\big(d, D(d)\big)$ the closure of the closable operator considered in the previous result.}

\vskip0.2truecm\noindent
One of the conclusion of the above result reads ${\overline H_\phi}=\frac{1}{2}d^*\circ d$, i.e. the $L^2$-generator has the aspect of a "generalized Laplacian" composed of a "generalized divergence" operator $d^*$ and a "generalized gradient" operator $d$. In other words, the operator $d$ (essentially the derivation $\partial$) is a differential square root of the $L^2$-generator.

\medskip
The next result shows that in the noncommutative space $C(\mathbb{G})$, the elements of the dense subalgebra $\mathcal{A}$ have a {\it noncommutative Lipschitz property}.  Below we denote by $\mathsf{A}{\widehat \otimes}H_\pi$ the projective tensor product of the C$^*$-algebra $\mathsf{A}$ and the Hilbert space $H_\pi$ which is
the completion of the algebraic tensor product of $A$ and $H_\pi$ with respect
to the norm
$$ \|x\|_{A{\widehat \otimes}H_\pi} := \inf \{\sum \|a_i\|_A \|\xi_i\|_{H_\pi},
\, \mbox{where} \,\, x= \sum a_i \otimes \xi_i\},$$
see \cite{grothendieck55} or \cite{ryan02}.

\begin{proposition} \label{prop_Dirac}
Let $\phi\in\A^\prime$ be a GNS-symmetric generating functional with associated Sch\"{u}rmann triple $((\pi, H_\pi ), \eta, \phi)$. Let us consider the Hilbert space $\mathcal{H}_\phi :=(L^2(\mathsf{A},h)\otimes H_\pi)\oplus L^2 (\mathsf{A},h)$ as a $\mathsf{A}$-bimodule under the commuting left and right actions
\[
\pi_{L} :=\rho_{L}\oplus \lambda_{L}\, ,\qquad \pi_{R} :=\rho_{R}\oplus \lambda_{R}\, .
\]
Consider also on $\mathcal{H}_\phi$ the self-adjoint operator
\[
D:=\left(
     \begin{array}{cc}
       0 & d \\
       d^* & 0 \\
     \end{array}
   \right)\, .
\]
Then the commutator $[D,\pi_L (a)]$ is bounded for all $a\in \A$ with norm bounded by
\[
\|[D,\pi_L (a)]\|\le \|\partial a\|_{\mathsf{A}{\widehat \otimes}H_\pi}.
\]
\end{proposition}
\begin{proof}

It follows from
\begin{eqnarray*}
[D,\pi_L (a)] &=& D \circ \pi_L(a) - \pi_L(a) \circ D \\
 &=& \left ( \begin{array}{cc}
 0 & d \\ d^* & 0 \end{array} \right )
\left ( \begin{array}{cc} \rho_L(a) & 0 \\ 0 & \lambda_L(a) \end{array} \right )
- \left ( \begin{array}{cc} \rho_L(a) & 0 \\ 0 & \lambda_L(a) \end{array} \right )
\left ( \begin{array}{cc} 0 & d \\ d^* & 0 \end{array} \right ) \\
&=& \left ( \begin{array}{cc} 0 & d \circ \lambda_L (a) - \rho_L(a) \circ d \\
d^*\circ \rho_L(a) - \lambda_L(a) \circ d^* & 0 \end{array} \right ) \\
&=& \left ( \begin{array}{cc} 0 & d \circ \lambda_L (a) - \rho_L(a) \circ d \\ - \big( d \circ \lambda_L (a^*) - \rho_L(a^*) \circ d\big)^* & 0 \end{array} \right ) \\
\end{eqnarray*}
that $[D,\pi_L (a)]$ is bounded for all $a\in \A$ if and only if $d \circ \lambda_L (a) - \rho_L(a) \circ d$ is bounded for all $a\in \A$. To check that the latter is actually the case, let us observe that, for $b\in\mathcal{A}$, we have
\begin{eqnarray*}
\big( d \circ \lambda_L (a) - \rho_L(a) \circ d \big) i_h (b)
&=& d(i_h	(ab)) - \rho_L(a) (\partial b)\\
&=& \partial(ab) - \rho_L(a) (\partial b)
= \rho_R(b) (\partial a)\, .
\end{eqnarray*}
For any presentation $\partial a = \sum_{k=1}^na_k\otimes \xi_k\in \mathsf{A}\otimes H_\pi$ we then have
\begin{eqnarray*}
\| \rho_R(b) (\partial a)\|_{L^2(\mathsf{A},h)\otimes H_\pi}&=& \| (\lambda_R\otimes {\rm id}_{H_\pi}) (\partial a)\|_{L^2(\mathsf{A},h)\otimes H_\pi}\\
&=&\|\sum_{k=1}^n a_k b\xi_h \otimes \xi_k\|_{L^2(\mathsf{A},h)\otimes H_\pi}\\
&\le & \sum_{k=1}^n \|a_k b\xi_h\|_{L^2(\mathsf{A},h)} \|\xi_k\|_{ H_\pi}\\
&\leq&  \|i_h(b)\|_{L^2 (A,h)} \sum_{k=1}^n \|a_k\|_A\cdot\|\xi_k\|_{H_\pi}\, .
\end{eqnarray*}
Optimizing among all presentations $\partial a = \sum_{k=1}^na_k\otimes \xi_k\in \mathsf{A}\otimes H_\pi$ we get
\[
\| \rho_R(b) (\partial a)\|_{L^2(\mathsf{A},h)\otimes H_\pi}\le  \|i_h (b)\|_{L^2 (A,h)}\cdot  \|\partial a\|_{A{\widehat \otimes}H_\pi}\qquad a,b\in\A
\]
so that
\[
\| d \circ \lambda_L (a) - \rho_L(a) \circ d\|\le \|\partial a\|_{A{\widehat \otimes}H_\pi}\qquad a\in\A\, .
\]
Finally notice that, setting $T_a := d \circ \lambda_L (a) - \rho_L(a) \circ d $, we have $\|T_a\| \le \|\partial a\|_{A{\widehat \otimes}H_\pi}$,
\[
[D,\pi_L (a)] =  \left ( \begin{array}{cc} 0 & T_a \\ -T_a^* & 0 \end{array} \right )
\]
and
\[
|[D,\pi_L (a)]|^2 =  \left ( \begin{array}{cc} T_a T_a^* & 0 \\ 0 & T_a^* T_a \end{array} \right ),
\]
so that
\[
\|[D,\pi_L (a)]\|= \|T_a\|\, ,\qquad a\in\A\, .
\]
\end{proof}

\begin{theorem}
\label{thm_spectral_triple}
Consider, on the Hopf $^*$-subalgebra $\A$ of a compact quantum group $\mathbb{G}=(\mathsf{A},\Delta)$, the GNS-symmetric generating functional $\phi\in\A^\prime$ with Sch\"{u}rmann triple $((\pi, H_\pi ), \eta, \phi)$.
\par\noindent
Consider also the GNS-symmetric, hermitian convolution generator $L_{\phi} :\A\rightarrow\A$ and its closed extension $(H_\phi ,D(H_\phi))$ on the space $L^2 (A,h)$, characterized by
\[
H_\phi (i_h(a)):=-i_h(L_{\phi} a)
\]
on its core $i_h(\A)=\A\xi_h\subset D(H_\phi)$.
\par\noindent
If the spectrum of $(H_\phi ,D(H_\phi))$ is discrete and considering the representation of $\mathsf{A}=C(\mathbb{G})$ constructed above
\[
\pi_{L} :=\rho_{L}\oplus \lambda_{L} :\mathsf{A}\rightarrow B(\mathcal{H}_\phi)  \qquad\mathcal{H}_\phi :=(L^2(\mathsf{A},h)\otimes H_\pi)\oplus L^2 (\mathsf{A},h)\, ,
\]
we have that $(\A ,D, (\pi_L ,\mathcal{H}_\phi))$ is a (possibly kernel-degenerate) spectral triple in the sense that
\begin{itemize}
\item $[D,\pi_L (a)]$ is a bounded operator for all $a\in\A$,
\item $D$ has discrete spectrum on the orthogonal complement of its kernel.
\end{itemize}
\end{theorem}
\begin{proof}
By construction
\[
D^2=\left(
     \begin{array}{cc}
       dd^* & 0 \\
       0 & d^* d \\
     \end{array}
   \right)\, ,
\]
so that the spectrum of $D^2$ is the union of the spectra of $dd^*$ and $d^* d$.
Since these two operators are unitarily equivalent on the orthogonal complement
of their kernels and zero belongs to the spectrum of $d^* d$, the spectrum of
$D^2$ coincides with the spectrum of $2H_\phi$, by Proposition
\ref{prop_derivation}. Since, by assumption, the spectrum of $H_\phi$ is
discrete we have that the spectrum of $D^2$, hence the one of $D$, are discrete
too on the orthogonal complement of their kernels. This result, together with
Theorem \ref{prop_Dirac} allows us to conclude the proof.
\end{proof}
The fact that the kernel of the Dirac operator $D$ may be infinite dimensional
is a variation with respect to the original definition of spectral triple given
in \cite{connes94}, due to the definition of $D$ as an antidiagonal matrix. To
construct the associated K-homology invariants this fact has to be taken into
account, for example using the methods developed in Section 3 of \cite{CGIS12}.

\section{Two classical examples: commutative and cocommutative CQGs}
\subsection{Algebras of functions on compact groups}
Let $G$ be a compact  Lie group and let $C(G)$ denotes the commutative
C${}^*$-algebra of all continuous functions on $G$. Then $C(G)$ is a compact
quantum group with the comultiplication
$$\Delta: C(G) \to C(G)\otimes C(G) \cong C(G\times G), $$
defined by
\begin{equation*} \label{comm_comultipliaction}
 \Delta (f)(s,t)=f(s t), \quad f\in C(G), \; s,t \in G.
\end{equation*}
 The counit and the antipode are defined on the dense $*$-subalgebra $C_c(G)$
generated by the coefficients of arbitrary continuous finite-dimensional
representation $\pi$, i.e. functions $\pi_{ij}: G \to \C$, and they are given by
$$\e (f)=f(e), \quad S(f)(x)=f(x^{-1}),\quad f\in C_c(G).$$

 This is a general example of a commutative compact quantum group, in the sense
that if $\mathsf{A}$ is the algebra of continuous functions on a compact quantum
group which is commutative as a C${}^*$-algebra, then there exists a unique
compact group $G$ such that $\mathsf{A}$ is isomorphic to $C(G)$ with coproduct
corresponding to the classical one given above (cf. \cite[Theorem
1.5]{woronowicz87}).

 The quantum group $C(G)$ is cocommutative if and only if the group $G$ is
commutative. It is always of Kac type, i.e. $S^2=\id$. This implies that the
modular automorphism group is trivial and that the Haar state is tracial (see
Remark \ref{rem_kac_type}), and so the notions of GNS- and KMS-symmetry coincide
(see Remark \ref{rem_kms_gns}).

\medskip
 The generating functionals of L\'evy processes in $G$ are classified by
Hunt's formula as follows (cf. \cite{liao04}). Let $\{X_1,X_2,\ldots, X_d\}$ be
a fixed basis of the Lie algebra $\mathfrak{g}$ associated to the Lie group $G$
and let $x_1,x_2,\ldots x_d \in C_c^\infty (G)$ be the local coordinates
associated to this basis, i.e. $X_i=\frac{\partial}{\partial x_i}$ at the
neutral element $e$.
Then an arbitrary generating functional $\phi$ is of the form
\begin{eqnarray*}
 \phi (f) &=& \sum_{i=1}^d c_i X_i f(e) + \frac12 \sum_{j,k=1}^d a_{jk} X_j X_k f(e)\\
&+& \int_{G\setminus \{e\}} \left( f(g)-f(e) -\sum_{i=1}^d x_i (g) X_i f(e) \right) \nu (dg)
\end{eqnarray*}
 for twice differentiable $f$. Here $c_i, a_{jk}$ are real constants,
$(a_{jk})_{j,k=1}^d$ is a positive definite symmetric matrix and the measure
$\nu$ on $G$ satisfies
$$ \nu(\{e\})=0, \quad  \int_U \sum_{i=1}^d x_i^2 {\rm d}\nu <\infty, \quad \nu (G\setminus U)<\infty$$
 for any neighborhood $U$ of $e$ in $G$. The first term in the decomposition
above is called the \emph{drift}, whereas the second one is called the
\emph{diffusion}. The measure $\nu$ is called \emph{L\'evy measure}.

 The GNS-symmetric processes correspond to functionals with no drift part and
symmetric L\'evy measures, i.e. $\nu(E)=\nu(E^{-1})$ for measurable $E$ (see
\cite[Proposition 4.3]{liao04}, where such processes are called \emph{invariant
under the inverse map}).

 The characterisation of $\ad$-invariant processes (called \emph{conjugate
invariant} in \cite{liao04}) depends on the particular group structure. The two
extreme cases are abelian Lie groups and simple Lie groups. In the first case,
as observed in Section \ref{sec_adinv}, the adjoint action is trivial and all
functionals are $\ad$-invariant. If the Lie group is simple and connected, then
the adjoint action $\ad (f) (x,y)= f(xyx^{-1})$ has trivial kernel. Then the
L\'evy measure of an $\ad$-invariant process must be \emph{conjugate-invariant}
(or \emph{central}, cf. \cite{applebaum10}), that is $\nu(gEg^{-1})=\nu(E)$ for
all measurable $E$. Moreover, the drift part vanishes and the diffusion part is
(up to a constant) the Beltrami-Laplace operator on $G$ (see \cite[Propositions
4.4, 4.5]{liao04}). In the case the Dirichlet form reduces to the Dirichlet integral on $G$
\[
\E[a]=\int_G\, |\nabla a(g)|^2\, dg
\]
defined on the Sobolev space $H^{1,2}(G)$ of functions having square integrable gradient and the derivation is just the gradient operator.
We refer to \cite{liao04} for details on this topic.
\vskip0.2truecm

\subsection{C${}^*$-algebra of a countable discrete group}

Let $\Gamma$ be a countable discrete group and let $\ell^2(\Gamma)$ denote the Hilbert
space of all square-summable functions on $\Gamma$. The space $\ell^2(\Gamma)$
is spanned by the orthonormal basis $\{\delta_g : g\in \Gamma\}$, where as usual
$\delta_g(h)=1$ if $g=h$ and $\delta_g(h)=0$ otherwise. Then each element $g \in
\Gamma$ defines the linear operator $\lambda_g : \ell^2(\Gamma) \to \ell^2(\Gamma)$
by the formula $$ \lambda_g(\delta_h) = \delta_{gh}, \quad h\in \Gamma.$$
Each $\lambda_g$ is a unitary operator and the mapping $g\to \lambda_g$ is
called the left regular unitary representation of the Hilbert space $\Gamma$ on
$\ell^2(\Gamma)$.

The closure of the $*$-algebra generated by $\{\lambda_g: g\in \Gamma\}$ in
$B(\ell^2(\Gamma))$ is denoted by $C^*_r(\Gamma)$ and called the \emph{reduced
C${}^*$-algebra} or the \emph{group algebra} of $\Gamma$. One can also define
the \emph{universal} C${}^*$-algebra of the group, denoted by $C^*_u(\Gamma)$,
by taking the direct sum of all cyclic representations of $\Gamma$ (universal
representation) instead of the left regular one. The two algebras are isomorphic
if and only if $\Gamma$ is amenable, cf. \cite{pedersen79}.

The mapping $\Delta$ defined by $\Delta(\lambda_g)=\lambda_g \otimes
\lambda_g$ extends (in a unique way) to a $*$-homomorphism from
$C^*_r(\Gamma)$ to $C^*_r(\Gamma)\otimes C^*_r(\Gamma)$ which preserves the
unit. The pair $(C^*_r(\Gamma), \Delta)$ is a compact quantum group. The linear
span $\A$ of $\{\lambda_g: g\in \Gamma\}$ in $B(\ell^2(\Gamma))$ is a $^*$-Hopf
algebra on which counit and antipode are defined by $\e(\lambda_g)=1$ and
$S(\lambda_g)=\lambda_{g^{-1}}$ respectively, for $g\in \Gamma$.

\medskip
The quantum group $C^*_r(\Gamma)$ is always cocommutative (i.e.\ the
comultiplication is invariant under the flip).  Moreover, each algebra of
continuous functions on a compact quantum group which is cocommutative  is
essentially of this form (there exists a unique discrete  group $\Gamma$ and
$*$-homomorphisms $C_u^*(\Gamma)\to \mathsf{A}\to C_r^*(\Gamma)$), see
\cite[Theorem 1.7]{woronowicz87}.  Cocommutativity implies that the adjoint
action is trivial: $\ad(a)=\1\otimes a$, $\ad_h(a)=a$ and all functionals are
$\ad$-invariant $\phi\circ \ad_h=\phi$.

\medskip
The algebra $C^*_r(\Gamma)$ is of Kac type so that the modular automorphism
group is trivial. The Haar state is a trace and, on generators, it is
explicitly given by $h(\delta_g)= 0$ for $g\neq e$ and $h(\delta_e)=1$. The GNS
Hilbert space $L^2 (C^*_r(\Gamma),h)$ can then be identified with $l^2
(\Gamma)$.

In this case the notions of GNS and KMS symmetry coincide, and $\phi$ is symmetric iff
$\phi(\lambda_g)=\phi(\lambda_{g^{-1}})$ for any $g\in \Gamma$.
 Moreover, symmetric generating functionals of L\'evy processes are in one-to-one
correspondence with (obviously continuous) positive, conditionally negative-type
functions
\[
d:\Gamma \to [0,\infty),\qquad d(g)=-\phi(\lambda_g), \quad g\in \Gamma
\]
(cf. \cite[Example 10.2]{cipriani+sauvageot03}). The associated Dirichlet form is given by
\[
\E[a]= \sum_{g\in\Gamma}\, d(g)|a(g)|^2 \qquad a\in l^2 (\Gamma)
\]
and the generator of the Markovian semigroup on $l^2 (\Gamma)$ is just the multiplication operator
\[
(H_\phi a)(g)=d(g)a(g),
\]
defined for those $a\in l^2 (\Gamma )$ such that the right hand side in square integrable.

\medskip
The derivation associated to the KMS symmetric generating functional $\phi$
(recall Section \ref{sec_derivation}) is given by $\partial(\lambda_g) =
\lambda_g \otimes \eta(\lambda_g)$ for $g\in \Gamma$, where $\eta$ is the
1-cocycle corresponding to $\phi$ in the Sch\"urmann triple $((\pi,D), \eta,
\phi)$.
Composing the 1-cocycle $\eta$ on the C$^*$-algebra $C^*_r(\Gamma)$ with the
left regular representation, one obtains the 1-cocycle
\[
c:\Gamma\to D,\qquad c(g)=\eta (\lambda_g)
\]
on the group $\Gamma$. In terms of this, the negative type function is given by
\[
d(g)=\|c(g)\|^2_D\, .
 \]
Identifying $l^2 (\Gamma)\otimes D$ with $l^2 (\Gamma ,D)$, one obtains that the derivation above reduces to the multiplication operator
\[
(\partial a)(g)=c(g)a(g)\qquad g\in\Gamma\, ,
\]
defined for all $a$ in the domain of the Dirichlet form.

The spectrum of the generator $H_\phi$ is discrete if and only if the
negative-type function $d$ is proper on $\Gamma$ (a condition which is met, for
example, for some length functions of finitely generated groups, see for
example \cite{cherix+cowling+jolissant+julg+valette01}). In these situations the
construction of a spectral triple shown in Theorem \ref{thm_spectral_triple}
applies.

\section{Example: free orthogonal quantum groups $O_N^+$}
\label{sec-On}

Let $N\ge 2$. The compact quantum group $(C_u(O_N^+),\Delta)$ is the universal
unital C${}^*$-algebra generated by $N^2$ self-adjoint elements $v_{jk}$,
$1\le j,k\le N$ subject to the condition that the matrix $V=(v_{jk})\in
M_N\otimes C_u(O_N^+)$ is a unitary corepresentation, i.e.\ that
\[
\sum_{\ell=1}^N v_{\ell j} v_{\ell k} =\delta_{jk} = \sum_{\ell =1}^N v_{j\ell} v_{k\ell}
\]
and
\[
\Delta(v_{jk}) = \sum_{\ell =1}^N v_{j\ell} \otimes v_{\ell k}
\]
for all $1\le j,k\le N$,
see \cite{vandaele+wang96,banica96}. The equivalence classes of the irreducible unitary corepresentations of this compact quantum group can be indexed by $\mathbb{N}$, with $u^{(0)}=\1$ the trivial corepresentation and $u^{(1)}=(v_{jk})_{1\le j,k\le N}$ the corepresentation whose coefficients are exactly the $N^2$ generators of $C_u(O_N^+)$ (this is also called the \emph{fundamental} corepresentation of $O_N^+$). The dense *-Hopf algebra ${\rm Pol}(O_N^+)$ associated to $O_N^+$, also called the *-algebra of polynomial on $O_N^+$, is the *-algebra generated by  $v_{jk}$, $1\le j,k\le N$. The compact quantum group
$O_N^+$ is
called the \emph{free orthogonal compact quantum group}. For $N>2$ it is not co-amenable, i.e.\
the Haar state of $O_N^+$ is not
faithful on $C_u(O_N^+)$, therefore we will study the Markov semigroups of
L\'evy processes on ${\rm Pol}(O_N^+)$ on the reduced $C^*$-algebraic
version $C_r(O_N^+)$ of $O_N^+$.

The compact quantum group $O^+_N$ is of Kac type, and therefore a generating
functional $\phi$ is KMS-symmetric if and only if it is GNS-symmetric, which is
the case if the characteristic matrices are symmetric, i.e.\ if
$\phi(u_{jk}^{(s)})=\phi(u_{kj}^{(s)})$ for all $s\in \mathbb{N}$ and $j,k$
running from 1 up to the dimension of the $s$th corepresentation.

Corollary \ref{cor-bijection} reduced the problem of classifying ad-invariant
generating functionals on a compact quantum group to the classification of
generating functionals on the subalgebra of central functions. For the free
orthogonal quantum group $O^+_N$ the algebra of central functions is
isomorphic the the C${}^*$-algebra of continuous functions on the interval
$[-N,N]$, cf.\ \cite[Corollary 4.3]{brannan2011}. Furthermore, the restriction
of the counit to this subalgebra is the evaluation of a function in a boundary point.

Let us begin by describing linear functionals which are positive on a given
interval and vanish in a given point.
\begin{proposition}\label{prop-levy-khinchin}
Denote by $\tau_x:C([0,1])\to\mathbb{C}$ the evaluation of a function in $x\in[0,1]$.
\begin{itemize}
\item[{\bf (a)}]
Suppose $0<x<1$. A linear functional
$\varphi:\mathbb{C}[x]\to \mathbb{C}$ with $\varphi(1)=0$ is positive on the cone
\[
K_x([0,1]) = \mathbb{C}[x]\cap C([0,1])_+ \cap {\rm ker}(\tau_x)
\]
if and only if there exist real numbers $a,b$ with $a\ge 0$ and a finite measure $\nu$ on $[0,1]$ with $\nu(\{x\})=0$ such that
\[
\varphi(f) = bf'(x) + a f''(x) + \int_0^1 \big(f(y)-f(x)-yf'(x)\big) \frac{\nu({\rm d}y)}{(y-x)^2}
\]
for all polynomials $f\in C([0,1])$.

 The triple $(a,b,\nu)$ is uniquely determined by $\varphi$. We will call
$(a,b,\nu)$ the \emph{characteristic triple} of the linear functional
$\varphi$.
\item[{\bf (b)}]
Suppose $x\in\{0,1\}$. Then a linear functional
$\varphi:\mathbb{C}[x]\to \mathbb{C}$ with $\varphi(1)=0$ is positive on the cone
\[
K_x([0,1]) = \mathbb{C}[x]\cap C([0,1])_+ \cap {\rm ker}(\tau_x)
\]
 if and only if there exist a real number $d$ with $d\ge 0$ if $x=0$, and $d\le
0$ if $x=1$, and a finite measure $\mu$ on $[0,1]$ with $\mu(\{0\})=0$ such that
\[
\varphi(f) = df'(x) + \int_0^1 \big(f(y)-f(x)\big) \frac{\mu({\rm d}y)}{y}
\]
for all polynomials $f\in C([0,1])$.

 The pair $(d,\mu)$ is uniquely determined by $\varphi$.  We will call $(b,\nu)$
the \emph{characteristic pair} of the linear functional $\varphi$.
\end{itemize}
\end{proposition}

\begin{proof}
{\bf (a)}
This is actually the classical L\'evy-Khinchin formula for L\'evy processes on
$\mathbb{R}$, see, e.g., \cite[Theorem 8.1]{sato99}, which can be viewed as
a special case of Hunt's formula \cite{hunt56}. Skeide \cite{skeide99} has
given a C${}^*$-algebraic proof which doesn't use the group structure, but works
for the $C^*$-algebra of continuous functions on a compact set, with a
character given by evaluation in a fixed point which has neighborhood with
Euclidean coordinates (i.e.\ smooth functions admit a Taylor expansion around
the fixed point).

{\bf (b)}
This is actually the classical L\'evy-Khinchin formula for subordinators,
cf. \cite[Theorem 21.5]{sato99}. We prove the formula for $x=0$, the case
$x=1$ follows easily by a change of variable $t\mapsto 1-t$.

By (a), since $\varphi$ has to be positive
also on the smaller cone given by polynomials that vanish in $x=0$ and which
are positive on $[-\varepsilon,1]$ for any $\varepsilon>0$, there exists a unique triple $(a,b,\nu)$ with $a,b\in \mathbb{R}$ with $a\ge 0$ and $\nu$ a finite measure on $[0,1]$, such that
\[
\varphi(f) = bf'(0) +a f''(0) + \int_0^1
\big(f(y)-f(0)-yf'(0)\big)\frac{\nu({\rm d}y)}{y^2}.
\]
For $n \in\mathbb{N}$ we set $g_n(y) = \frac{1}{n+1} y \sum_{k=0}^n (1-y)^k$. We
have $g_n\in K_0([0,1])$, $g'_n(0)=1$, and $g''_n(0)=-n$, therefore
\[
0\le \varphi(g_n) = b -na + \int_0^1 \big(g_n(y)-y\big)\frac{\nu({\rm d}y)}{y^2}
\]
for all $n\in\mathbb{N}$. The sequence $(g_n)_{n\in\mathbb{N}}$ is decreasing
and therefore we must have $a=0$. By monotone convergence we get
\[
b \ge \int_0^1 \big(y-g_n(y)\big)\frac{\nu({\rm d}y)}{y^2} \quad\stackrel{n\to\infty}{\longrightarrow}\quad \int_0^1
\frac{\nu({\rm d}y)}{y}
\]
which proves that the measure  $\frac{1}{y}\nu$ is finite. Putting
$\mu=\frac{1}{y}\nu$ and $d=b-\int_0^1 \frac{\nu({\rm d}y)}{y}\ge 0$, we get
the desired formula
\[
\varphi(f) = df'(0) + \int_0^1 \big(f(y)-f(0)\big) \frac{\mu({\rm d}y)}{y}.
\]
Conversely, since a polynomial $f$ which vanishes in $x=0$ and is
positive on $[0,1]$ has a positive derivative at $x=0$, it is clear that any
such functional is positive on $K_0([0,1])$. Uniqueness follows from (a).
\end{proof}

This result allows us to describe all ad-invariant generating functionals on
${\rm Pol}(O^+_N)$. This result can be considered as Hunt's formula for
ad-invariant L\'evy processes on the free orthogonal quantum group $O_N^+$.

Let us denote by ${\rm Pol}_0(O^+_N)$ the algebra of central polynomial
functions on $O^+_N$, see Eq.\ \eqref{eq_central_functions}. We will use the
same isomorphism between ${\rm
Pol}_0(O^+_N)$ and polynomials ${\rm Pol}([-N,N])$ as Brannan
\cite{brannan2011}. Recall that Banica \cite{banica96} showed that the
equivalence classes of irreducible unitary
corepresentations of $O_N^+$ can be labelled by non-negative integers and that
they satisfy the ``fusion rules''
\[
u^{(s)}\otimes u^{(t)} \cong u^{(|s-t|)}\oplus u^{(|s-t|+2)}\oplus \cdots
\oplus u^{(s+t)}
\]
for $s,t\in\mathbb{N}$. Since the trivial corepresentation
$u^{(0)}=\mathbf{1}$ has dimension $1$ and the fundamental corepresentation
$u^{(1)}=(v_{jk})_{1\le j,k\le N}$ has dimension $N$, one can show by
induction that the dimensions of the irreducible unitary corepresentations are
given by Chebyshev polynomials of the second kind,
$D_s = U_s(N)$. The conditional expectation $\widetilde{\rm ad}_h:{\rm
  Pol}(O_N^+)\to {\rm Pol}_0(O_N^+)$ onto the algebra of central functions is
therefore given by
\[
\widetilde{\rm ad}_h(u^{(s)}_{jk}) = \frac{1}{U_s(N)} \delta_{jk} \chi_s,
\]
where $\chi_s=\sum_{j=1}^{D_s} u^{(s)}_{jj}$ denotes the trace of $u^{(s)}$.

The fusion rules imply that the characters satisfy the three-term recurrence relation
\[
\chi_1 \chi_s = \chi_{s+1} + \chi_{s-1}
\]
for $s\ge 1$, we get the desired isomorphism ${\rm Pol}(O^+_N)_0\cong
{\rm Pol}([-N,N])$ by setting $\chi_s\mapsto U_s$ for $s\in \mathbb{N}$, where
$U_s$ denotes the $s^{\rm th}$ Chebyshev polynomial of the second kind, defined
by $U_0(x)=1$, $U_1(x)=x$, and $U_{s+1}(x) = x U_s(x)-U_{s-1}(x)$.

\begin{theorem}\label{thm-hunt-on+}
The ad-invariant generating functional on ${\rm Pol}(O^+_N)$ are of the form
\[
\hat{L}=L\circ \widetilde{\rm ad}_h
\]
with $L$ defined on ${\rm Pol}(O^+_N)_0\cong {\rm Pol}([-N,N])$ by
\[
Lf = -b f'(N) + \int_{-N}^N \frac{f(x) - f(N)}{N-x} {\rm d}\nu(x)
\]
where $b\ge0$ is a real number and $\nu$ is a finite measure on $[-N,N]$ with $\nu(\{N\})=0$..
\end{theorem}
\begin{proof}
This follows from Theorem \ref{thm-hatS} and Proposition
\ref{prop-levy-khinchin}.
\end{proof}

Using the discussion above, we can give a formula for the values of
ad-invariant generating functionals on the coefficients of the irreducible
unitary corepresentations of $O_N^+$.

\begin{corollary}\label{cor-L}
The ${\rm ad}$-invariant generating functional on ${\rm Pol}(O^+_N)$ given in
Theorem \ref{thm-hunt-on+} with characteristic pair $(b,\nu)$ acts on the
coefficients of unitary irreducible corepresentations of $O_N^+$ as
\[
 L\big(u_{jk}^{(s)}\big) = \frac{\delta_{jk}}{U_s(N)}\left( -b U'_s(N) +
\int_{-N}^N \frac{U_s(x)-U_s(N)}{N-x} \nu({\rm d}x)\right)
\]
for $s\in\mathbb{N}$, where $U_s$ denotes the $s^{\rm th}$ Chebyshev polynomial
of the second kind.
\end{corollary}

\begin{remark}
Since the characteristic matrices of $L$ are diagonal, we can read off the eigenvalues of $T_L$ from Corollary
\ref{cor-L}. Assume for simplicity $b=1$, $\nu=0$.
Then the eigenvalues of $T_L$ are given by
\[
\lambda_s = -\frac{U'_s(N)}{U_s(N)},\qquad s\in\mathbb{N},
\]
with multiplicities given by the square of the dimension $m_s=D_s^2=\big(U_s(N)\big)^2$ of $u^{(s)}$.

Recall that the ``spectral dimension'' $d_D$ of the associated spectral triple is, by definition (see \cite{connes04a}, \cite{ connes04b}), the abscissa of convergence of the zeta function $z\mapsto \mathcal{Z}_D (z):={\rm Tr}\, (|D|^{-z})$, initially defined for $z\in \mathbb{C}$ with ${\rm Re}\,  z>0$. It coincides with the infimum of all $d>0$ such that the sum $\sum_s m_s(-\lambda_s)^{-d/2}$ is finite. In the present situation of Corollary \ref{cor-L} and assuming $N=2$, $b=1$, $\nu=0$, we have $U_s(2)=s+1$, $U'_s(2)=\frac{s(s+1)(s+2)}{6}$,
\[
\lambda_s= -\frac{s(s+2)}{6}
\]
and finally $d_D=3$. This value of the spectral dimension agrees nicely with the known fact that $O_2^+$ is isomorphic to $SU_{-1}(2)$, see \cite{banica96}, and that $C(SU_{-1}(2))$ can be
realized by matrix-valued functions on the three-dimensional Lie group $SU(2)$, cf.\ \cite{zakrzewski91}. On the other hand, for $N>2$, we have
\begin{eqnarray*}
U_s(N) &=& \frac{q(N)^{s+1} - q(N)^{-s-1}}{q(N)-q(N)^{-1}}, \\
U'_s(N) &=& \frac{q'(N)}{q(N)}\, \frac{s\big(q(N)^{s+2} - q(N)^{-s-2}\big) -
  (s+2)\big(q(N)^s - q(N)^{-s}\big)}{\big(q(N)-q(N)^{-1}\big)^2} \\
&=& \frac{q'(N)}{q(N)}\left(s\frac{q(N)^{s+1} - q(N)^{-s-1}}{q(N)-q(N)^{-1}}-2\frac{q(N)^s - q(N)^{-s}}{\big(q(N)-q(N)^{-1}\big)^2}\right),
\end{eqnarray*}
with $q(N)=\frac{1}{2}(N+\sqrt{N^2-4})>1$,
$q'(N)=\frac{1}{2}\left(1+\frac{N}{\sqrt{N^2-4}}\right)>0$, and
\[
\lambda_s=-\frac{q'(N)}{q(N)}\left(s- 2\frac{q(N)^s - q(N)^{-s}}{\big(q(N)^{s+1} - q(N)^{-s-1}\big)\big(q(N)-q(N)^{-1}\big)}\right).
\]
Since $q(N)$ is bigger then 1 (and fixed), the term
$$\frac{q(N)^s - q(N)^{-s}}{\big(q(N)^{s+1} -
q(N)^{-s-1}\big)\big(q(N)-q(N)^{-1}\big)} \to 0\quad \mbox{as} \quad s\to
\infty.$$
This implies that the growth of the eigenvalues $\lambda_s$ (as a function of
$s$) is asymptotically linear $\lambda_s\cong -\frac{q'(N)}{q(N)}s$, while the
multiplicities $m_s = U_s(N)^2\cong \frac{q(N)^{2s}}{(1-q(N)^{-2})^2} $ grow
exponentially, therefore the sum
$$\sum_s m_s(-\lambda_s)^{-d/2} \cong \sum_s
\frac{q(N)^{2s}}{s^{d/2}} $$
can never converge, which means that $d_D=+\infty$.
\end{remark}

\section{Example: Woronowicz quantum group $SU_q(2)$}

 Let us fix $q\in (0,1)$. The compact quantum group $C(SU_q(2))$ is the
universal unital C${}^*$-algebra generated by $\alpha$ and $\gamma$ subject to
the following relations
\begin{eqnarray*}
 & \alpha^* \alpha+\gamma^* \gamma=1, \quad \alpha \alpha^*+q^2\gamma \gamma^*=1, \\
 & \gamma^* \gamma = \gamma \gamma^*, \quad \alpha\gamma = q \gamma \alpha, \quad  \alpha\gamma^* = q \gamma^* \alpha
\end{eqnarray*}
with the comultiplication extended uniquely to a unit-preserving $*$-homomorphism from the formulas
\begin{equation*}\label{suq2_comultiplication}
  \Delta(\alpha)=\alpha \otimes \alpha - q\gamma^* \otimes \gamma, \quad
\Delta(\gamma)=\gamma \otimes \alpha + \alpha^* \otimes \gamma.
\end{equation*}

 For $C(SU_q(2))$ the equivalence classes of irreducible unitary
corepresentations are indexed by non-negative half-integers $s\in \frac12
\mathbb{N}$ and are of dimension $n_s=2s+1$. For each
$u^{(s)}=(u_{jk}^{(s)})_{j,k}$ the indices $j,k$ run over the set $\{-s, -s+1,
\ldots, s-1,s\}$ (see eg. \cite{pusz+woronowicz00} for the detailed description
of $u^{(s)}$). Moreover, every corepresentation is equivalent to its
contragredient one and we have
\begin{equation}
 \label{eq-star}
S(u_{kj})=(u_{jk}^{(s)})^* = (-q)^{k-j}u_{-j,-k}^{(s)}.
\end{equation}

The quantum group is neither commutative nor cocommutative.
 The Woronowicz characters, the modular automorphism group, the unitary antipode
and the quantum dimension are the following (cf. \cite[Appendix
A1]{woronowicz87}:
\begin{eqnarray}
\label{suq2-form}
f_z(u_{jk}^{(s)})&=&q^{2jz} \delta_{jk}, \\ 
\label{suq2-sigma}
\sigma_{z}(u_{jk}^{(s)})
&=& f_{iz} \star u_{jk}^{(s)} \star f_{iz} =
q^{2iz(j+k)} u_{jk}^{(s)}, \\ 
\label{suq2-R}
R(u_{jk}^{(s)})
 &=& S(f_{\frac{1}{2}} \star u_{jk}^{(s)} \star f_{-\frac{1}{2}})
= q^{k-j}(u_{kj}^{(s)})^*,\\
\label{suq2-D}
D_s&=& \sum_{k=-s}^{s} f_1(u_{kk}^{(s)})= \sum_{k=-s}^{s} q^{2k}
=q^{-2s} [2s+1]_{q^2}.
\end{eqnarray}

\medskip
The following example describes the irreducible representations of $C(SU_q(2))$
and the related opposite representations (cf. Section \ref{sec_schurmann}).
\begin{example} \label{rep_opp}
 On $\mathsf{A}=C(SU_q(2))$ we have two families of irreducible $*$-representation indexed by $\theta\in [0,2\pi)$:
\begin{enumerate}
\item the 1-dimensional representations $\delta_\theta:\mathsf{A} \to\C$:
$$ \delta_\theta(\alpha)= e^{i\theta}, \quad \delta_\theta(\gamma)=0;$$
\item the infinitely-dimensional representations on a Hilbert space
$\rho_\theta:\mathsf{A} \to B(\ell^2)$:
$$ \rho_\theta(\alpha)e_n= We_n, \quad \rho_\theta(\gamma)e_n=e^{i\theta}q^ne_n,$$
where $(e_n)_{n\in \mathbb{N}}$ is the standard orthonormal basis of $\ell^2$
and $W$ is the weighted shift defined by $We_0=0$ and
$We_n=\sqrt{1-q^{2n}}e_{n-1}$ for $n\geq 1$.
\end{enumerate}
We check directly that
$$\delta_\theta^\op = \delta_{-\theta} \quad \mbox{ and } \quad \rho_\theta^\op = \rho_{\pi+\theta}.$$

Indeed, we note first that $R(\alpha)=\alpha^*$ and $R(\gamma)=-\gamma$. So  for $\delta_\theta$ we have
$$\delta_\theta^\op (\alpha)\bar{1} = \overline{\delta_\theta(R(\alpha^*))1} =
\overline{\delta_\theta(\alpha)1} =\overline{e^{i\theta}1} =  e^{-i\theta} \bar{1}=\delta_{-\theta} (\alpha)\bar{1} $$
and
$$\delta_\theta^\op (\gamma)\bar{1} = \overline{\delta_\theta(R(\gamma^*))1} =
\overline{-\delta_\theta(\gamma^*)1} =0=\delta_{-\theta} (\gamma)\bar{1}. $$

Similarly, for $\rho_\theta$ we compute
$$\rho_\theta^\op (\gamma) \bar{e}_n = \overline{-\rho_\theta(\gamma^*)e_n} =\overline{-e^{-i\theta}q^ne_n} = e^{i(\pi+\theta)}q^n\bar{e}_n=\rho_{\pi+\theta} (\gamma)\bar{e}_n.$$

$$\rho_\varphi^\op (\alpha) \bar{e}_n = \overline{\rho_\varphi(\alpha)e_n} =\overline{We_n} =  \sqrt{1-q^{2n}} \bar{e}_n=\rho_{\pi+\theta} (\alpha)\bar{e}_n.$$
\end{example}

\subsection{GNS-symmetric generators}
\label{example}

We first describe a generic GNS-symmetric functional on $SU_q(2)$ and provide an
example of an unbounded generating functional.

\begin{proposition}\label{generic_symmetric_suq2}
A hermitian functionals $\phi$ on $SU_q(2)$ defined by
\begin{equation} \label{ex_gns_suq2}
 \phi(u_{jk}^{(s)}) = c_{s,j}\delta_{jk}
\end{equation}
with real constants $(c_{s,j})_{s\in \frac12\mathbb{N}, -s\leq j\leq s}$ is
GNS-symmetric.

Reciprocally, any GNS-symmetric generator $\phi$ on $SU_q(2)$ must be of the
form \eqref{ex_gns_suq2}
and it is hermitian if and only if the constants satisfy the supplementary symmetry condition: $c_{s,j}=c_{s,-j}$ for all $s$ and $j$.

\end{proposition}

\begin{proof}
 We calculate explicitly that $\phi$ of the form \eqref{ex_gns_suq2} is
invariant under the antipode:
 $$ \phi\circ S (u_{jk}^{(s)}) = \phi ((u_{kj}^{(s)})^*) = \overline{\phi
(u_{kj}^{(s)}))} = \overline{c_{s,j}}\delta_{jk}=c_{s,j}\delta_{jk}=\phi
(u_{jk}^{(s)}).$$

Conversely, suppose that $\phi\circ S=\phi$, then also  $\phi\circ S^2=\phi$.
On $SU_q(2)$ we have  $S^2(u_{jk}^{(s)})=q^{2(j-k)}u_{jk}^{(s)}$, thus
$\phi(u_{jk}^{(s)})=\phi\circ S^2(u_{jk}^{(s)})=q^{2(j-k)}\phi(u_{jk}^{(s)})$ and $\phi$ can have non-zero values only on the diagonal. By Remark \ref{remark_symmetry}, all $c_{s,j}$ must be real.

Finally, $c_{s,j} = \bar{c}_{s,j}= \overline{\phi(u_{jj}^{(s)})}$ and $ c_{s,-j}=\phi(u_{-j,-j}^{(s)})=\phi((u_{jj}^{(s)})^*)$ from which the last part follows.
\end{proof}

\subsection{Unbounded GNS-symmetric generator}

Let $\pi$ be the $*$-representation of $SU_q(2)$ on $\ell^2(\mathbb{N}\times
\mathbb{Z})$ given by
\begin{eqnarray*}
\pi (\alpha) e_{k,n} &= & \sqrt{1-q^{2k}}e_{k-1,n} \; (k\geq 1), \quad \pi (\alpha) e_{0,n}=0,\\
\pi (\alpha^*) e_{k,n} &= & \sqrt{1-q^{2k+2}}e_{k+1,n} \; (k\geq 0), \\
\pi (\gamma) e_{k,n}&= & q^k e_{k,n-1}, \quad \pi (\gamma^*) e_{k,n} = q^k e_{k,n+1},
\end{eqnarray*}
where $\{e_{k,n}; k\geq 0, n\in \mathbb{Z}\}$ is the standard orthonormal basis
of $\ell^2(\mathbb{N}\times \mathbb{Z})$. For a fixed  $0< \lambda <1$ let us
consider a Poisson type generator
$$\phi_\lambda (a) =\langle v_\lambda, (\pi -\e)(a) v_\lambda\rangle \quad \mbox{with} \,\, v_\lambda = \sum_{k=0}^\infty \lambda^k e_{k,0}.$$
The related cocycle $\eta_\lambda(a)= (\pi -\e)(a) v_\lambda$ is uniquely determined by the value on $\alpha^*$ (see \cite{schuermann+skeide98}), where it equals
\begin{eqnarray*}
\eta_\lambda(\alpha^*)&=& \sum_{k=0}^\infty \lambda^k (\pi -\e)(\alpha^*) e_{k,0} = \sum_{k=0}^\infty \lambda^k \big(\sqrt{1-q^{2k+2}}e_{k+1,0}- e_{k,0}\big) \\
&=& \sum_{k=1}^{\infty} \lambda^{k-1} \sqrt{1-q^{2k}}e_{k,0}- \sum_{k=0}^{\infty} \lambda^k e_{k,0}\\
&=& -e_{0,0}+\sum_{k=1}^{\infty} \big( \lambda^{k-1} \sqrt{1-q^{2k}}-\lambda^{k}  \big)e_{k,0}.
\end{eqnarray*}
Note that $\eta_\lambda(\alpha^*)\in H$ since
\begin{eqnarray*}
\|\eta_\lambda(\alpha^*)\|^2&=&
1+ \sum_{k=1}^{\infty} \big| \lambda^{k-1} \sqrt{1-q^{2k}}-\lambda^{k}  \big|^2
\leq 1+4\sum_{k=1}^{\infty} \lambda^{2(k-1)} <+\infty.
\end{eqnarray*}

Define also a cocyle $\eta_\infty$ by its value on $\alpha^*$:
$$\eta_\infty (\alpha^*) = -e_{0,0}+\sum_{k=1}^{\infty} (\sqrt{1-q^{2k}}-1)e_{k,0} =
-\sum_{k=0}^{\infty} (1-\sqrt{1-q^{2k}})e_{k,0}. $$
We shall show that $\eta_\infty (\alpha^*)\in H$ and that $ \eta_\lambda(\alpha^*) \rightarrow  \eta_\infty (\alpha^*)$ in $H$ when $\lambda \to 1^-$. For the first part, we check directly that
\begin{eqnarray*}
\|\eta_\infty (\alpha^*)\|^2&=&
\sum_{k=0}^{\infty} \big(1- \sqrt{1-q^{2k}}\big)^2
= \sum_{k=0}^{\infty} \frac{q^{4k}}{(1+ \sqrt{1-q^{2k}})^2}
< \frac{1}{1-q^4} <+\infty.
\end{eqnarray*}
Next, we show the convergence:
\begin{eqnarray*}
\lefteqn{\|\eta_\lambda (\alpha^*)-\eta_\infty(\alpha^*)\|^2
=
\| \sum_{k=1}^{\infty} \big( \lambda^{k-1} \sqrt{1-q^{2k}}-\lambda^{k}  - \sqrt{1-q^{2k}}+1)e_{k,0} \|^2 }\\
&=&
\sum_{k=1}^{\infty} \big|  (1-\sqrt{1-q^{2k}})(1-\lambda^{k-1})+ \lambda^{k-1}(1-\lambda) \big|^2 \\
&=&
\sum_{k=1}^{\infty} (1-\sqrt{1-q^{2k}})^2(1-\lambda^{k-1})^2 + (1-\lambda)^2\sum_{k=1}^{\infty}  \lambda^{2(k-1)} \\
& + & 2(1-\lambda)\sum_{k=1}^{\infty} (1-\sqrt{1-q^{2k}})(1-\lambda^{k-1})\lambda^{k-1}.
\end{eqnarray*}
Note that $1-\lambda^{k-1}= (1-\lambda)(\lambda^{k-2}+\ldots + 1) \leq (k-1)(1-\lambda)$. This implies
\begin{eqnarray*}
\lefteqn{\|\eta_\lambda (\alpha^*)-\eta_\infty(\alpha^*)\|^2
= }\\ &\leq &
(1-\lambda)^2\sum_{k=1}^{\infty} (k-1)^2 \frac{q^{4k}}{(1+\sqrt{1-q^{2k}})^2} + \frac{(1-\lambda)^2}{1- \lambda^{2}} \\
& + & 2(1-\lambda)^2\sum_{k=1}^{\infty} \frac{q^{2k}}{1+\sqrt{1-q^{2k}}}(k-1)\lambda^{k-1}
\\ &\leq &
(1-\lambda)^2\sum_{k=1}^{\infty} (k-1)^2 q^{4k} + \frac{1-\lambda}{1+ \lambda}
+ 2(1-\lambda)^2\sum_{k=1}^{\infty} q^{2k}(k-1)
\end{eqnarray*}
and we see that each term tends to 0 when $\lambda \to 1^-$.

\medskip
Let us now define by $\phi_\infty$ the functional related to the cocycle $\eta_\infty$ by the formula
$$\phi_\infty (ab)=\langle \eta_\infty (a^*),\eta_\infty(b) \rangle, \quad a,b \in \ker \e$$
 with the additional conditions that  $\phi_\infty(\1)=0$ and that the `drift'
part is zero (which remains to say that
$\phi_\infty(\alpha)=\phi_\infty(\alpha^*)\in \R$). This way $\phi_\infty$ is
uniquely determined on the whole of $\A={\rm Lin} \{\1, \alpha-\alpha^*, K_2\}$,
where $K_2$ is the linear span of products of two elements from $\ker \e$ (cf.
\cite{schuermann+skeide98}). By the Schoenberg correspondence, if
well-defined, $\phi_\infty$ is a generating functional of a L\'evy process.

To see that $\phi_\infty$ is  well-defined, we can check that on $K_2$ the
functional is just a limit of functionals related to $\eta_\lambda$. Indeed, if
$a,b\in \ker \e$ then
$$\phi_\infty (ab)= \langle \eta_\infty (a^*),\eta_\infty(b) \rangle
=\langle \lim_{\lambda \to 1^-} \eta_\lambda (a^*),\lim_{\mu \to 1^-} \eta_\mu(b) \rangle$$
and since both limits exist we have
 $$\phi_\infty (ab)=  \lim_{\lambda \to 1^-} \langle \eta_\lambda (a^*),\eta_\lambda (b) \rangle = \lim_{\lambda \to 1^-} \phi_\lambda (ab) = \lim_{\lambda \to 1^-} \langle v_\lambda, (\pi -\e)(ab) v_\lambda\rangle. $$
We conclude that
\begin{equation} \label{phi_on_k2}
 \phi_\infty (a)=\lim_{\lambda \to 1^-} \langle v_\lambda, (\pi -\e)(a) v_\lambda\rangle \quad \mbox{for} \quad a\in K_2.
\end{equation}

Our aim now is to show that $\phi_\infty$ is GNS-symmetric and unbounded.

\begin{proposition}
 The functional $\phi_\infty$ is GNS-symmetric.
\end{proposition}

\begin{proof}
 By Proposition \ref{generic_symmetric_suq2}, it is enough to show that
$\phi_\infty$ vanishes on the non-diagonal coefficients of the corepresentations
$u^{(s)}$, $s\in \frac12\mathbb{N}$. These coefficients are of the form
$b_{m,n}:=\alpha^m p(\gamma^*\gamma) \gamma^n$ for $m,n \in \mathbb{Z}$, (with
the notation $a^{-n}=(a^*)^n$ for $n>0$ and $p$ denoting a polynomial, see
\cite{pusz+woronowicz00}) and they are off diagonal iff $n\neq 0$. So it is
enough to check that $\phi_\infty$ vanishes on $\alpha^m (\gamma^*\gamma)^k
\gamma^n$ ($n\neq 0$).

We first observe that $\gamma, \gamma^*\in K_2$. Indeed, the relation
$\alpha\gamma = q \gamma \alpha$ together with $\gamma, \alpha-\1\in \ker \e$
imply that
$$ \gamma = \frac{q}{1-q}\gamma (\alpha-\1) - \frac1{1-q}(\alpha-\1)\gamma \in
K_2.$$
Therefore an element $\alpha^m (\gamma^*\gamma)^k \gamma^n$ belongs to $K_2$
provided $k \neq 0$ or $n\neq 0$. So the formula \eqref{phi_on_k2} can be
applied
\begin{eqnarray*}
  \phi_\infty (\alpha^m (\gamma^*\gamma)^k \gamma^n)=\lim_{\lambda \to 1^-}
\sum_{p,r=0}^\infty \lambda^{p+r} \langle e_{p,0}, \pi (\alpha)^m
\pi(\gamma^*\gamma)^k \pi(\gamma)^n) e_{r,0}\rangle \quad
\end{eqnarray*}
 Since $\pi(\gamma)$ and $\pi(\gamma^*)$ move (down and up, respectively) the
second index of the basis vectors $e_{k,n}$ and since none of $\pi(\alpha)$,
$\pi(\alpha^*)$ and $\pi(\gamma^*\gamma)$ move the second index, we immediately
see that if $n\neq 0$ then $ \pi (\alpha)^m \pi(\gamma^*\gamma)^k \pi(\gamma)^n
e_{r,0} \in \C \cdot e_{r+m,n}$, which is orthogonal to $e_{p,0}$ for any $m,k$
and $p,q$. So the sum under the limit equals to 0 and thus $\phi_\infty$ is of
the form \eqref{ex_gns_suq2}.
\end{proof}

\medskip

\begin{proposition}
 The functional $\phi_\infty$ is unbounded.
\end{proposition}

\begin{proof}
We shall show that $|\phi_\infty({\alpha^*}^{m}{\alpha}^{m})| \to +\infty$ when $m\to +\infty$. Since
$$\|{\alpha^*}^m{\alpha}^{m}\|_\mathsf{A}\leq \|\alpha\|^{2m}_\mathsf{A}\leq
1,$$
this will imply that $\phi_\infty$ is unbounded.

Observe first that
\begin{eqnarray*}
{\alpha^*}^m\alpha^m &=& {\alpha^*}^{(m-1)}(\alpha^*\alpha){\alpha}^{m-1}
={\alpha^*}^{(m-1)}(\1-\gamma^*\gamma){\alpha}^{m-1}\\
&=& {\alpha^*}^{(m-1)}{\alpha}^{m-1}(\1-q^{-2(m-1)}\gamma^*\gamma)
\end{eqnarray*}
and by induction
$${\alpha^*}^m\alpha^m = (\1-\gamma^*\gamma)(\1-q^{-2}\gamma^*\gamma)\ldots (\1-q^{-2(m-1)}\gamma^*\gamma), \quad m\geq 1.$$
Applying the standard formula from the $q$-calculus (cf. \cite[Equation (0.3.5)]{koekoek+swarttouw94}:
$$ (a;q)_n = \sum_{k=0}^n \frac{(q^2;q^2)_n}{(q^2;q^2)_k(q^2;q^2)_{n-k}} q^{k(k-1)} (-a)^k$$
we arrive at
\begin{eqnarray*}
{\alpha^*}^m\alpha^m-\1 &=& (\1-\gamma^*\gamma)\ldots (\1-q^{-2(m-1)}\gamma^*\gamma)-\1 \\
&=& \sum_{k=1}^{m} (-1)^k \frac{(q^2;q^2)_{m}}{(q^2;q^2)_k(q^2;q^2)_{m-k}} q^{k(k-1)} (\gamma^*\gamma)^k.
\end{eqnarray*}
We see that each term under the sum contains $\gamma^*\gamma$ and thus belongs to $K_2$.

Now that we have proved that ${\alpha^*}^m\alpha^m-\1\in K_2$, we can
apply the formula \eqref{phi_on_k2} to calculate the value of $\phi_\infty$ on
${\alpha^*}^m\alpha^m$. Namely,
\begin{eqnarray*}
 \lefteqn{\phi_\infty ({\alpha^*}^m\alpha^m) = \phi_\infty ({\alpha^*}^m\alpha^m-\1) }\\
&=& \lim_{\lambda \to 1^-} \sum_{j,k=0}^\infty \lambda^{j+k} \langle e_{j,0}, (\pi -\e)({\alpha^*}^m\alpha^m-\1) e_{k,0}\rangle \\
&=& \lim_{\lambda \to 1^-} \sum_{j,k=0}^\infty \lambda^{j+k} \langle e_{j,0},  \big[(I_H-\pi(\gamma^*\gamma))\ldots (I_H-q^{-2(m-1)}\pi(\gamma^*\gamma))-I_H \big] e_{k,0}\rangle \\
&=& \sum_{k=0}^\infty \big( (1-q^{2k})\ldots (1-q^{2k-2m+2})-1 \big) \\
&=& - \sum_{k=0}^{m-1} 1 - \sum_{k=m}^\infty \big(1- (1-q^{2k})\ldots (1-q^{2k-2m+2})\big).
\end{eqnarray*}
We finally note that the infinite sum is non-negative, and so
$$|\phi_\infty ({\alpha^*}^m\alpha^m)| = m-1+ \sum_{k=m}^\infty \big(1- (1-q^{2k})\ldots (1-q^{2k-2m+2})\big) \geq m-1. $$
\end{proof}

\subsection{KMS-symmetry}

 In case of $C(SU_q(2))$ it is easy to check that a hermitian $\phi$ is
KMS-symmetric iff for each $s\in \frac12 \mathbb{N}$ the matrix
$\phi_q^{(s)}=[q^j\phi (u^s_{jk})]$ is hermitian.
 Moreover, if a functional $\phi$ is hermitian and KMS-symmetric, then the values
of  $\phi$ on the corepresentation matrix $u^{(s)}$ are determined by the values
$\phi(u_{jk}^{(s)})$ for $|k|\leq j$ and the conditions
 $$ \phi(u_{-j,-k}^{(s)})=(-q)^{j-k}\overline{\phi(u_{jk}^{(s)})} \qquad
\mbox{and} \qquad  \phi(u_{kj}^{(s)})=q^{j-k}\overline{\phi(u_{j,k}^{(s)})}. $$
 Below we provide an example of a KMS-symmetric generator which is not
GNS-symmetric.

\begin{example}[KMS-symmetric generator which is not GNS-symmetric] \label{ex_KMS_not_sym}
 Let us consider the Poisson type generating functional on $SU_q(2)$
 $$\phi(a)=\langle e_k, (\rho_\theta-\e)(a)e_k\rangle,$$
 where $\rho_\theta$ is the infinite-dimensional representation of $C(SU_q(2))$
on $\ell^2$, described in Example \ref{rep_opp}, $\theta\in [0, 2\pi)$, and
$e_k$ is the $k$th standard orthonormal basis vector of $\ell^2$.

If (and only if) $\theta=\frac{\pi}{2}$ or $\theta=\frac{3\pi}{2}$, then $\phi$ is KMS-symmetric.
Indeed, by Theorem \ref{thm_schurmann_triples} and Example \ref{rep_opp} the
condition for the generating functional to be KMS-symmetric, $\phi(a)=\phi\circ
R(a)$, reduces to
\begin{equation} \label{poisson_KMS_ex}
\langle e_k, (\rho_\theta(a)-\rho_{\theta+\pi}(a^*)) e_k \rangle = \e(a)-\e(a^*)
\end{equation}
for any $a\in \A$. For $a=\gamma$ the left hand side of \eqref{poisson_KMS_ex} is
 $$\langle e_k, \big(\rho_\theta(\gamma)-
\rho_{\theta+\pi}(\gamma^*)\big)e_k\rangle = \langle e_k, (e^{i\theta}
-e^{-i(\theta+\pi)})q^ke_k\rangle = (e^{i\theta}+e^{-i\theta}) q^k,$$
which equals $|\e(\gamma)|^2=0$ only when $\theta=\frac{\pi}{2}$ or
$\theta=\frac{3\pi}{2}$.
 For such $\theta$ we check by a direct calculation that the equation
\eqref{poisson_KMS_ex} holds true for each element of the form $a=(\alpha^*)^l
\gamma^m (\gamma^*)^n$ (such elements form a linear basis of $\A$).

The Poisson generator related to $\rho_\theta$ with $\theta\in \{
\frac{\pi}{2}, \frac{3\pi}{2}\}$ is not GNS-symmetric since
$S(\gamma)=-q\gamma$ and $q\neq 1$ imply $\phi(\gamma)\neq \phi\circ
S(\gamma)$.

\medskip

For $k=0$ and $\theta=\frac{\pi}{2}$ we can calculate explicitly the values of
the generating functional \eqref{poisson_KMS_ex}. Namely, using the explicit
formula for the coefficient of the corepresentations (cf.
\cite[(B.19)]{pusz+woronowicz00}), the Vandermonde summation formula
(cf.\,(0.5.9) in \cite{koekoek+swarttouw94}) and the standard $q$-transformation
(cf.\,(0.2.14) therein), we get
$$\phi(u_{jk}^{(s)})=\left \{\begin{array}{ll}
-1, & j=k \\
i^{2s} q^{(s-j)(s+j+1)} -\delta_{0,j}, & j\geq 0, k=-j \\
i^{-2s} q^{(s-j)(s+j+1)-2j} , & j< 0, k=-j \\
0, & \mbox{otherwise.}
\end{array}  \right. $$
In particular, it has non-zero entries only on the diagonal and the anti-diagonal.

The eigenvalues of the matrix $\phi^{(s)}=[\phi(u_{jk}^{(s)})]$, $\phi(a)=\langle e_0, (\rho_{\frac{\pi}{2}}-\e)(a)e_0\rangle$ are:
$$
 \lambda_{j}^{+} = -(1+q^{(s-j)(s+j+1)+j}), \quad \lambda_{j}^{-} = -(1-q^{(s-j)(s+j+1)+j})
 $$
with $j=\frac12,\frac32, \ldots ,\frac{k}{2}$ when $s=\frac{k}{2}$ ($k\in 2\mathbb{N}+1$) or
$$
 \lambda_0=-1+(-1)^s q^{s(s+1)}, \quad \lambda_{j}^{+} = -(1+q^{(s-j)(s+j+1)+j}), \quad \lambda_{j}^{-} = -(1-q^{(s-j)(s+j+1)+j})
$$
for $j=1,3, \ldots , k$ when $s=k$ ($k\in \mathbb{N}$).
\end{example}

\subsection{$\ad$-invariance}
 We already noted that $\ad$-invariance is a strong constraint on the
functional. Namely, it is necessarily a multiple of the identity on each of the
corepresentation matrices, in particular of diagonal form. A comparison of
this notion with that of GNS-symmetry (suggested by Proposition
\ref{generic_symmetric_suq2}), shows that $\ad$-invariant generating functionals
of a L\'evy process on $C(SU_q(2))$ are necessarily GNS-symmetric.

\begin{corollary}
 Let $\phi$ be a functional on $SU_q(2)$. If $\phi$ is \ad-invariant and
hermitian, then $\phi$ is GNS-symmetric.
\end{corollary}

\begin{proof}
 By Proposition \ref{prop_adinv_form}, $\phi$ is of the form
$\phi(u_{jk}^{(s)})=c_s \delta_{jk}$. By hermiticity, $ c_s =\phi(u_{jj}^{(s)})
=\phi \big( (u_{-j,-j}^{(s)})^* \big) = \overline{\phi  (u_{-j,-j}^{(s)})}=
\bar{c}_s$ and we conclude by Proposition \ref{generic_symmetric_suq2} that
$\phi$ is GNS-symmetric.
\end{proof}

The following example shows that the map $\phi \to \phi_{\rm ad}$ preserves
neither hermiticity nor positivity.
\begin{example} \label{ex_adinvariance_notherm}
 Let $\phi$ be the functional on $SU_q(2)$ defined by $\phi (\alpha)=e^{it}$,
$\phi(\alpha^*)=e^{-it}$ and zero otherwise, where $t \not\in 2\pi \mathbb{Z}$.
Then $\phi_{\ad}$ is ad-invariant and $\phi_{\ad}(\alpha)=\phi_\ad(\alpha^*)=
(1-q^2)^{-1} (e^{it} + q^2 e^{-it})$, so it is not hermitian.
\end{example}


\section*{Acknowledgements}

U.F.\ and A.K.\ would like to thank the Politecnico di Milano for the hospitality during their visits. 
They also thank the GDRE GREFI GENCO for the financial support for these visits. The research was supported by the MIUR project PRIN 2009 N. 2009KNZ5FK-004 "Equazioni alle derivate parziali degeneri o singolari:
metodi metrici".
\par\noindent
The research was mainly conducted during A.K.'s post-doctoral stay in
Laboratoire de Math\'ematiques de Besancon, supported by the Region
Franche-Comt\'e.
\par\noindent
A.K.\ was also partially supported by the Polish National Research Center (NCN)
post-doctoral internship no. DEC-2012/04/S/ST1/00102. U.F.\ was supported by the
ANR Project OSQPI (ANR-11-BS01-0008).
\par\noindent
We are also indebted to Makoto Yamashita for bringing the reference \cite[Lemma
4.1]{voigt2011} to our attention.
\par\noindent
We wish to thank the referee for his/her scrupulous work.


\end{document}